\documentclass[letterpaper,12pt]{article}

\usepackage{microtype,booktabs}
\usepackage{amsmath,amssymb,amsfonts,amsthm,bm}
\usepackage{authblk}
\usepackage{mathrsfs}
\usepackage{indentfirst}
\usepackage[pdftex]{color,graphicx}
\usepackage[pdftex,bookmarks,unicode,colorlinks,linkcolor=blue]{hyperref}
\usepackage{tikz}
\usepackage{subfigure}
\usetikzlibrary{arrows, decorations.pathmorphing, backgrounds, positioning, fit, petri, automata,cd}
\usepackage{rotating}
\numberwithin{equation}{section}

\usepackage{todonotes}

\title{\Large A parameterization method for quasi-periodic systems with noise: computation of random invariant tori}

\date{}
\author[$\dag$]{\normalsize Pingyuan Wei}
\author[$\ddag$]{Lei Zhang\thanks{Corresponding author.}}
\affil[$\dag$]{\small School of Mathematics, Southeast University, Nanjing 211189, China $\&$ Beijing International Center for Mathematical Research, Peking University, Beijing 100871, China. (Email:pwei@seu.edu.cn)}
\affil[$\ddag$]{School of Mathematical Sciences, Dalian University of Technology, Dalian 116023, China.
(Email:lzhang@dlut.edu.cn)}

\newtheorem{thm}{Theorem}[section]
\newtheorem{theorem}{Theorem}[section]

\newtheorem{corollary}[thm]{Corollary}
\newtheorem{lem}[thm]{Lemma}
\newtheorem{lemma}[thm]{Lemma}
\newtheorem{pro}[thm]{Proposition}

\theoremstyle{remark}

\newtheorem{remark}[thm]{Remark}
\theoremstyle{definition}
\newtheorem{definition}[thm]{Definition}

\newcommand{\ZZ}{\mathbb{Z}}

\newcommand{\NN}{\mathbb{N}}
\newcommand{\RR}{\mathbb{R}}
\newcommand{\CC}{\mathbb{C}}
\newcommand{\TT}{\mathbb{T}}
\newcommand{\PP}{\mathbb{P}}

\begin{document}
\maketitle

\begin{abstract}
This work is devoted to studying normally hyperbolic invariant manifolds (NHIMs) for a class of quasi-periodically forced systems subject to additional stochastic noise. These systems can be understood as skew-product systems. The existence of NHIMs is established by developing a parameterization method in random settings and applying the Implicit Function Theorem in appropriate Banach spaces. Based on this, we propose a numerical algorithm to compute the statistics of NHIMs and Lyapunov exponents.
\end{abstract}

\tableofcontents

\section{Introduction}

Transition State Theory (TST) has long been regarded as one of the most significant frameworks for understanding and estimating chemical reaction rates \cite{Miller1993, Waalkens2007}. Interesting quick historical overviews can be found in \cite{Hernandez1993, Truhlar1996}. In autonomous systems, the transition state is commonly modeled as a normally hyperbolic invariant manifold (NHIM). Reaction rates in such systems can be derived from the geometric properties of NHIM under specific conditions \cite{Uzer2002, Mackay2014}. For systems subject to time-varying external forces, the geometric structures of TST are known to exist in various cases \cite{Lehmann2000, CravenBH14, CravenBH15} even though they become time-dependent, and  can be determined by effective algorithms \cite{Junginger2016, Revuelta2017, ZhangL18}. Recently, there has been growing interest in studying transition states in chemical reactions under external stimuli \cite{Craven2015JCP, Junginger2017}. These stimuli, which include light, electric fields, thermal variations, and photo-induction, enable control over the rate and pathways of transformations from reactants to products. Such approaches have important implications for harnessing mechanical action in applications \cite{Leigh2003, Fletcher2005, Zazza2013, Valsson2016}. A key aspect of these systems is the natural emergence of stochastic driving forces (e.g., stochastic noise). This makes it more appropriate to consider
the state of the systems evolving through random or stochastic differential equations rather than deterministic ones \cite{Oksendal2013, Duan15}. This shift in perspective not only aligns with the inherent randomness in such systems but also opens exciting new avenues for understanding and predicting reaction dynamics in noisy environments.

The specific problem addressed in this paper is the study of the transition state in oscillating potential models that incorporate dissipation and are subject to stochastic influences. These models are represented as random skew-product dynamical systems, where the base map is a rigid rotation on a torus (quasi-periodic systems), and they naturally arise in the modeling of chemical reactions under external perturbations \cite{Craven2015JCP}. More precisely, let $x(t)$(resp. $v(t)$) denote the position (resp. velocity) of a chemical reaction system at time $t$, evolving in $\RR^d$, under the influence of the following forces:
\begin{itemize}
      \item[1.] (Quasi-Periodically) Potential Force: The force $- \nabla U(x-E(t))$ arises from a potential $U$ with a time-dependent center $E(t)$. In this paper, $E(t)$ is considered as a quasi-periodic function with frequencies $\theta=(\theta_1,\cdots,\theta_m)$ given by:   
      \begin{equation}\label{eq:1}
      E(t)=\hat{E}(\theta_1,\cdots,\theta_m), 
      \end{equation}
      where $\hat{E}$ is a a mapping from $\TT^m$ (the $m$-dimensional torus) to $\RR^d$. A natural example is
      \begin{equation}\label{eq:1-1}
          E(t)=\sum_{i=1}^d \sin(\alpha_it+\beta_i)
      \end{equation}
      with $\alpha_i\in\RR \backslash Q$ (i.e., irrational frequencies) and $\beta_i\in\RR$. Following \cite{CravenBH15,Craven2015JCP}, a common choice for the potential is
  \begin{equation}
    \label{eq:2}
    U(x)=-\frac{1}{2}x^2-\frac{1}{4}\delta x^4, 
  \end{equation}
  where $\delta$ is an anharmonic coefficient. In this paper, we assume that $\delta<0$, making $U(x)$ a double-well potential.
  \item[2.] Dissipative Force: The dissipative force $-\gamma v$, where $\gamma>0$ represents the strength of dissipation.  
  \item[3.] Stochastic Driving Force: The stochastic influence is represented as $\varepsilon\xi_t(\omega)$,  where $\xi_t(\omega)$ is an independent stochastic process on a probability space $( \Omega , \mathscr{F} , \PP)$ with values in $\RR^d$, and the parameter $\varepsilon$ controls the strength of the noise. This stochastic driving force can take various forms, such as continuous bounded process, Gaussian white noise and non-Gaussian L\'evy noise. For simplicity and practical relevance, we focus on the Gaussian white noise, modeled as ${\xi}_t=\dot{W}_t$, the formal time derivative of a standard Brownian motion (or Wiener process) \cite{Duan15}. This choice is common in applications. However, the methods developed here may also be applicable to other types of stochastic influences.
  \end{itemize} 
  The dynamics under consideration can be described by the following Langevin equation, which is interpreted as a stochastic differential equation (SDE) when ${\xi}$ is a white noise process:
  \begin{equation}   
  \label{eq:model}
  \begin{split}
  \left \{
  \begin{array}{ll}
    \dot{x}_t&=v_t,\\
    \dot{v}_t&=-\gamma v_t-  \nabla U (x_t-E(t))+\varepsilon\xi_t,
  \end{array}
   \right.
    \end{split}
  \end{equation}
  with initial value $x(0)=x_0$ and $v(0)=v_0$. For a detailed mathematical discussion of this equation, refer to Section 2.

In the absence of oscillations ($E(t)=0$) and noise ($\varepsilon=0$), the system is dissipative, it will converge to the phase points where the Hamiltonian $H(x,v)=\frac{1}{2}v^2+U$ attains local minima. The primary question in this case is to characterize the boundary of the region of initial conditions that eventually converge to one of the Hamiltonian's minima. In such scenarios, the transition state corresponds to a hyperbolic fixed point in phase space. When the system is driven by a time-dependent oscillating external potential  ($E(t)\neq0$), the dissipative analysis breaks down. A common approach to handle this case is to study an extended system. In this extended framework, the transition state becomes a torus (or, in certain cases, a circle) in phase space \cite{CravenBH15, CravenBH14, Canadell2015, ZhangL18}. When stochastic noise is introduced ($\varepsilon>0$), it may compensate the energy loss due to dissipation. This leads to  some fascinating phenomena, see \cite{Roberts1990, Wu2001} for details. In this work, we take both oscillations and noise in account and focus on studying the transition state. Our goal  is to extend the methods of \cite{ZhangL18} to incorporate the effects of noise. We demonstrate that the invariant manifolds described in \cite{ZhangL18} persist under the addition of stochastic noise. Furthermore, we would like to present effective methods to compute and analyze various statistical features of the resulting objects.

The fact that we are dealing with a stochastic system leads to some complicated notation
and bookkeeping \cite{Arnold98, Duan15}. The main difficulty as well as the crux of the extension lies in handling the noise term and demonstrating the required measurability properties. Furthermore, since a stochastic dynamical system is inherently nonautonomous, the concept of an invariant manifold must be extended. To understand the invariance of a random manifold, we interpret it as follows: each orbit that starts within the manifold stays within it sample-wisely, module
the change of sample due to noise (e.g., via time shifts, as defined in the next section). Thus, the transition state (which, in the deterministic case, is typically an NHIM) becomes an object that depends on the realization of the noise.

It should be pointed out that various approaches have been developed to study perturbation results for (normally hyperbolic) invariant manifolds. Among the most notable are Hadamard's graph transform method and the Liapunov-Perron method; see, for example, \cite{Arnold98, DuanLS03, DuanLS04, Liji2013} and the references therein. In contrast, the method we present here differs from these classical approaches. It indeed extends recent results by \cite{ZhangL18, Cheng2019}, and is inspired by studies on the parameterization method, originally introduced in \cite{CFdelaL2003a, CFdelaL2003b, HdlL06a, HdlL06b, Haro}. The core idea of the parameterization
method we will use is to describe the (random) invariant manifold as a (random) embedding map of a reference manifold into the phase space.  The invariance property
is thus described by a functional equation (which we called random invariance equation) for the embedding map. In other words, the parameterization method reduces the problem of establishing the existence of a (random) invariant manifold to solving specific functional equations. It is particularly well-suited for designing algorithms to compute invariant manifolds, especially in quasi-periodically forced systems. Furthermore, compared to other numerical techniques, such as those based on non-stationary normal form theory \cite{normalform1, normalform2}, the parameterization method offers a key advantage: computations are performed directly in the original physical coordinates so that the results retain a clear physical interpretation.

This paper is organized as follows. In the preliminary Section 2, we first review some basic concepts on SDEs, and then introduce the techniques of random transformations and autonomization. This leads to the fact that we are working with a random (quasi-periodic) skew-product system. Furthermore, we introduce the definitions of random invariant manifolds and random invariant equations, which form the foundation of our study. In Section 3, we address the existence and persistence of (normally hyperbolic) random invariant tori. By carefully choosing appropriate Banach spaces and defining transfer operators, we establish and prove the main result, Theorem \ref{Existence-Persistence}, based on the elementary Implicit Function Theorem. In Section 4, we study the expansions of random invariant tori in terms of the small parameter $\varepsilon$, and then develop numerical algorithms to estimate the random invariant tori as well as the corresponding Lyapunov exponents. Applications to the Langevin equation \eqref{eq:model} with both continuous bounded process and white noise are presented in Section 5. 

\section{Preliminaries}\label{Setups}

\subsection{A stochastic differential equation (SDE)}
Mathematically, the model \eqref{eq:model} is indeed an SDE of the form
\begin{equation}
\label{Cauchy-SDE}
d{z}= Az dt + B(z,t)dt+\varepsilon\sigma d{W}_t,
\end{equation}
where: (with $O$, $I$ the $d\times d$ zero and unit matrices respectively)
\begin{itemize}
    \item $z=(x,v)^T\in \RR^{2d} $ is the state-space point combining position $x$ and velocity $v$;
    \item $A=\begin{small}\begin{pmatrix}  O & I \\ I & -\gamma I \end{pmatrix} \end{small}$ is a $2d\times 2d$ hyperbolic matrix with $\gamma>0$, and  $B(z,t)=\big(O,-E(t)+(x-E(t)^3\big)^T$ is a nonlinear function with $E$ given in \eqref{eq:1-1};
    \item  
    $\sigma=\big(O,I\big)^T$ is a degenerate diffusion coefficient, $\varepsilon$ is a small parameter and ${W}_t$ is a $d$-dimensional standard Brownian motion. 
\end{itemize}
This SDE is interpreted in the integral form as
\begin{equation}
\label{Interpreted-SDE}
{z}(t)= z_0+\int_0^t\big(Az+B(z,t)\big)ds+\varepsilon\int_0^t \sigma d{W}_s.
\end{equation}
Here, the first integral is a standard Lebesgue integral, representing the \textit{drift} term, while the second is an Itô integral with respect to the Brownian motion, characterizing the \textit{diffusion} term in the stochastic system. Note that $E(t)$ is a bounded function and $U(x)$ is a bounded-below polynomial growing at infinity like $|x|^{2l}$ with $l$ a positive integer. Under these conditions, the existence and uniqueness of global solutions to the Langevin system \eqref{eq:model} or \eqref{Cauchy-SDE} are guaranteed, as established in \cite{Mattingly2002,SongXie2020}. For further details on the well-posedness and related properties of SDEs, we refer to \cite{Mao2007,Duan15}.

To effectively apply the parameterization method, it may be necessary to consider a modified (or ``cut-off") system. This approach may ensure that both the drift and diffusion functions possess desirable properties, such as global Lipschitz continuity, and that the solution flows exhibit specific regularities. Importantly, the cut-off system is designed to align with the original model except for low-probability events, thereby preserving the essential dynamics. A detailed discussion of the cut-off system for \eqref{eq:model} is provided in Appendix A.

\subsection{Random transformations: Converting SDEs to RDEs}

Before introducing the concept of random invariant manifolds, it is beneficial to first convert an SDE system into a system of random differential equations (RDEs). This conversion is necessary because the existing theory on random invariant manifolds applies to ``\textit{cocycles}" (also known as \textit{random flows} or \textit{random dynamical systems}), and RDEs are indeed differential equations with random coefficients which are easily seen to generate cocycles \cite{Arnold98,Duan15}.

To facilitate the study of dynamical systems in stochastic settings, we briefly review some fundamental concepts. Notably, the choice of the probability space $( {\Omega} , {\mathscr{F}}, \PP)$ in the definition of Brownian motion is arbitrary. However, it is often more appropriate and convenient to work with the \textit{canonical probability space} (also known as the \textit{classic Wiener space}; see \cite{Arnold98}). We now consider 
\begin{equation}\label{Wiener space}
{\Omega}=\{\omega: \omega(\cdot)\in C (\RR,\RR^d), \omega(0)=0 \},
\end{equation}
endowed with the open compact topology so that ${\Omega}$ is a Polish space, $\mathscr{F}$ the corresponding $\sigma$-algebra, and $\PP$ the Wiener measure. In this framework, a sample path of the Brownian motion can be treated as a point in ${\Omega}$, and we write $W_t(\omega)=\omega(t)$. The \textit{Wiener shift} $\{\Phi_t\}$ is defined as a mapping on the canonical probability space: 
\begin{equation}\label{Wiener shift}
\Phi_{t}:{\Omega}\to{\Omega},\;\;\Phi_{t}\omega(\cdot)=\omega(t+\cdot)- \omega(t),
\end{equation}
for each fixed $t\in\RR$. It is well known that $\Phi_t$ is measurable and satisfies that $\Phi_0=\text{Id}$, $\Phi_{t}\Phi_{s}=\Phi_{t+s}$ for $\forall t,s\in\RR$, and $\PP(A)=\PP(\Phi_t^{-1}A)=\PP(\Phi_{-t}A)$ for $\forall A\in\mathscr{F}$. Furthermore, it is closely related to the Brownian motion $W_t$ or the noise ``$dW_t$" by its definition. The equation \eqref{Wiener shift} means that
$W_s(\Phi_t\omega)=W_{t+s}(\omega)-W_t(\omega)\approx dW_t(\omega)$, when $s$ is infinitesimally small.
\\

In this paper, rather than considering the entire space ${\Omega}$, we only need to concentrate on a $\{\Phi_t \}_{t\in\RR}$-invariant subset ${\Omega}^\ast\subset{\Omega}$ of full measure based on the lemma below. For simplicity, we will adopt the same notation and identify ${\Omega}^\ast$ and ${\Omega}$. To appropriately handle the noise term $\sigma dW$ and establish that \eqref{Cauchy-SDE} generates a random dynamical system, we introduce the following preliminary considerations. Specifically, we examine the linear SDE:
\begin{equation}
  \label{eq: linear}
 d\zeta=-\zeta dt+\sigma dW,
\end{equation}
whose solution corresponds to the well-known \textit{Ornstein-Uhlenbeck process}. The following results, which will be essential for our analysis, are derived from Lemma 2.1 in \cite{DuanLS03}.

\begin{lemma} \label{OU}
  (i) There exists a $\{\Phi_t \}_{t\in\RR}$-invariant subset ${\Omega}^\ast\subset{\Omega}$ of full measure with sublinear growth, that is, $\lim_{t\to\pm\infty}{|\omega(t)|}/{|t|}=0$ for a.e. $\omega\in{\Omega}^\ast$.
 (ii) For $\omega\in{\Omega}^\ast$ the random variable 
\begin{equation}
  \label{eq: OU-rv}
\zeta(\omega)=-\int_{-\infty}^0e^{s}\sigma \omega(s)ds
\end{equation}
exists and generates a unique solution of \eqref{eq: linear} given by
\begin{align}
    \label{eq: OU-rp}
    (t,\omega)\to \zeta(\Phi_t\omega)
    =&-\int_{-\infty}^0e^{s}\sigma\Phi_t\omega(s)ds=-\int_{-\infty}^te^{s-t}\sigma\omega(s)ds\notag\\
    =& e^{-t}\zeta(\omega)-\int_0^te^{s-t}\sigma\omega(s)ds.
\end{align}
(iii) For $\omega\in{\Omega}^\ast$ the solution $\zeta(\Phi_t\omega)$ to \eqref{eq: linear} satisfies the following properties:
\begin{itemize}
    \item[a.] It is continuous with respect to $t$;
    \item[b.] It is stationary: $\zeta(\Phi_{t+s}\omega)=\zeta(\Phi_{t}\Phi_{s}\omega)$ for any $t,s\in\mathbb{R}$;
    \item[c.] $\lim_{t\to\pm\infty}{|\zeta(\Phi_t\omega)|}/{|t|}=0$;
    \item[d.] $\lim_{t\to\pm\infty}\frac{1}{t}\int_0^t\zeta(\Phi_s\omega)ds=\mathbb{E}\zeta(\Phi_t\omega)=0$;
    \item[e.] $\mathbb{P}[|\zeta(\Phi_t\omega)|< M]\geqslant 1-\frac{\mathbb{E}|\zeta(\Phi_t\omega)|}{M}$ for $M>0$ and $\mathbb{E}|\zeta(\Phi_t\omega)|<\infty$.
\end{itemize}
\end{lemma}

\begin{remark}
The sublinear growth property of $\omega(t)$ as well as $\varepsilon\zeta(\Phi_t\omega)$ plays a critical role in controlling the behavior of the noise terms and the solutions of the linearized system \eqref{eq: linear}. Specifically, it helps ensure that random perturbations do not overwhelm the deterministic structure of the system. This property is vital for proving the stability and hyperbolicity of random invariant manifolds, as it effectively bounds the long-term impact of noise on the system's dynamics.
\end{remark}

We now introduce the random transformations
\begin{equation}
  \label{eq: transform}
T(z,\omega)=z-\varepsilon\zeta(\omega),\;\;\;T^{-1}(z,\omega)=z+\varepsilon\zeta(\omega),
\end{equation}
for $z\in\RR^{2d}$ and $\omega\in {\Omega}$, where $\zeta(\omega)$ is given in \eqref{eq: OU-rv}.
\begin{lemma} \label{RDE}
Under the transformation ${Z}(t)=T({z}(t),\Phi_t\omega)$, the SDE \eqref{Cauchy-SDE} is equivalent to the following RDE:
\begin{align}
  \label{eq: random}
\frac{d}{dt}{Z}=A{Z}+{B}({Z}+\varepsilon\zeta(\Phi_t\omega),t)+\varepsilon (A+I)\zeta(\Phi_t\omega),
\end{align}
where the initial state ${Z}(0)={z}(0)=z_0$, and $\zeta(\Phi_t\omega)$ is a continuous Ornstein-Uhlenbeck process given in \eqref{eq: OU-rp}.
\end{lemma}

\begin{proof}
Substituting ${z}(t)={Z}(t)+\varepsilon\zeta(\Phi_t\omega)$ into the SDE \eqref{Cauchy-SDE} and noting that $\zeta(\Phi_t\omega)$ satisfies the linear SDE \eqref{eq: linear}, we obtain \eqref{eq: random} through straightforward calculations. Moreover, the equivalence between \eqref{Cauchy-SDE} and \eqref{eq: random} follows from the invertibility of the random transformation.
\end{proof}

In contrast to the SDE \eqref{Cauchy-SDE}, no stochastic differential appears in RDE \eqref{eq: random}. This equation has a unique solution for each $\omega\in\Omega$, and the solution can be expressed as
\begin{equation}
  \label{eq: Z}
Z(t)=z_0+\int_0^t A{Z}+{B}({Z}+\varepsilon\zeta(\Phi_s\omega),s)+\varepsilon (A+I)\zeta(\Phi_s\omega)ds.
\end{equation}
Therefore, the solution mapping generates a cocycle
\begin{equation}
  \label{eq: Z-mapping}
F : \RR^{2d}\times\RR\times\Omega \to \RR^{2d},\;\;F (z_0,t,\omega) =Z_t(z_0,\omega),
\end{equation}
 which is $(\mathscr{B}(\RR^{2d})\otimes \mathscr{B}(\RR)\otimes \mathscr{F},\mathscr{B}(\RR^{2d}))$-measurable such that
\begin{align}
F(z_0,0,\omega)=&z_0, \notag\\
F(z_0,t+s,\omega)=&F\big(F(z_0,t,\omega),s,\Phi_t\omega\big), \notag
\end{align}
for $t,s\in\RR$, $\omega\in\Omega$ and $z_0\in\RR^{2d}$. With the inverse conversion, the mapping 
\begin{equation}
  \label{eq: z-mapping}
\widehat{F}:=T^{-1}(\cdot,\Phi_t\omega)\circ F(T(z,\omega),t,\omega)
\end{equation}
is thus defined as the cocycle for the original SDE system \eqref{Cauchy-SDE}.


\begin{remark}
    We note that such a conversion from SDEs to RDEs is often possible but not always straightforward \cite{Arnold98,Imkeller2001,Duan15}. Similar transformations have been widely used to study random manifolds for certain stochastic systems; see, for instance, \cite{DuanLS03,DuanLS04}. Once random invariant manifolds are obtained for the RDE, they can be translated back to the original SDE through the inverse transformation; see Lemma \ref{sNHIM} in the next subsection.
\end{remark}

\subsection{Random invariant manifolds}\label{section invariance equation}
We recall that a multifunction $\mathcal{M}=\{\mathcal{M}(\omega)\}_{\omega\in \Omega}$, consisting of nonempty closed sets $\mathcal{M}(\omega)$, $\omega\in \Omega$, contained in $\RR^{2d}$ (or, more generally, a separable Banach space) is called a \textit{random set} if
$$
\omega \mapsto \inf_{z^{\prime}\in \mathcal{M}(\omega)} \|z-z^{\prime} \|
$$
is a random variable for any $z\in \RR^{2d}$. If each set $\mathcal{M}(\omega)$ is a manifold, then $\mathcal{M}$ is referred to as a \textit{random manifold}.

\begin{definition} 
 \label{defn: RIM}
(Random invariant manifolds) A random manifold $\mathcal{M}$ is called a \textit{random invariant manifold} for a cocycle $F$ if
\begin{equation}
  \label{eq:invariance random}
  F\big(\mathcal{M}(\omega),t,\omega\big)=\mathcal{M}(\Phi_t\omega),
\end{equation}
for $t\in\RR$, and $\omega\in\Omega$.
\end{definition} 
\begin{remark}
The geometric interpretation of \eqref{eq:invariance random} closely resembles the deterministic case: for each sample variable, the range of $\mathcal{M}$ remains invariant under the (random) flows of $F$. In this context, the manifold itself becomes a random variable, and it is possible to study statistical properties such as expected values of several events such as the manifold going trough a point, deviations of the above and so on. This formulation of the invariant manifold problem has become standard in the literature \cite{Arnold98}; see also \cite{Duan15}.
\end{remark}


Later in Definition \ref{defn:NHRIM}, we will further provide the definition of normally hyperbolic random invariant manifolds. Utilizing the conjugacy between the SDE \eqref{Cauchy-SDE} and the RDE \eqref{eq: random} via the random transformation \eqref{eq: transform}, we arrive at the following result.
\begin{lemma}\label{sNHIM}
Let $\mathcal{M}$ be the random invariant manifold for the RDE \eqref{eq: random}. Then 
$$
\widehat{\mathcal{M}}=\big\{T^{-1}\big(\mathcal{M}(\omega),\omega\big)\big\}_{\omega\in\Omega}=\Big\{\mathcal{M}(\omega)+\varepsilon\int_{-\infty}^0e^{-As}\sigma dW_s(\omega)\Big\}_{\omega\in\Omega}
$$
 is a random invariant manifold for the original SDE \eqref{Cauchy-SDE}. Furthermore, if $\mathcal{M}$ is normally hyperbolic, then so is $\widehat{\mathcal{M}}$.
\end{lemma}
\begin{proof}
Considering the relationship between $F$ and $\widehat{F}$ provided in \eqref{eq: z-mapping}, we have
\begin{align}
\widehat{F}\big(\widehat{\mathcal{M}}(\omega),t,\omega\big)
=&T^{-1}(\cdot,\Phi_t\omega)\circ F\big(T(\widehat{\mathcal{M}}(\omega),\omega),t,\omega\big)\notag\\
=&T^{-1}(\cdot,\Phi_t\omega)\circ F\big(\widehat{\mathcal{M}}(\omega),t,\omega\big)\notag\\
=&T^{-1}\big(\widehat{\mathcal{M}}(\Phi_t\omega),\Phi_t\omega\big) =\widehat{\mathcal{M}}(\Phi_t\omega),\notag
\end{align}
for $t\in\RR$ and $\omega\in\Omega$. Observe that $t\to \zeta(\Phi_t\omega)$ exhibits a sublinear growth rate (see Lemma \ref{OU}). The random transform $T^{-1}$ preserves the hyperbolicity property. Thus we are done with the proof.
\end{proof}

\subsection{The parameterization method and random invariance equations}

We highlight that one of the main novelties of our work is the inclusion of a quasi-periodic function $E(t)$ in the model. To address this, we aim to generalize the concept of the parameterization method \cite{HdlL06a,HdlL06b,Haro} and reformulate the geometric problem of random invariant manifolds as a functional analysis problem.

By introducing $d\theta=\alpha dt$ to the RDE \eqref{eq: random}, we consider the natural modification of the random flow:
$$F : \RR^{2d}\times \mathbb{T}^m\times\RR\times\Omega \to \RR^{2d},\;\;F (z_0,\theta,t,\omega) =Z_t(z_0,\theta,\omega)|_{\theta=\alpha t}.$$ Abusing notation, we denote by $F(z,\theta,\omega)=F(z,\theta,1,\omega)$ the time-1 map of the evolution. This leads us to search for invariant manifolds for the following system:
\begin{align}\label{discrete-stochastic-system} 
\bar{Z}={F}({Z},\theta,\omega),\;\;\bar{\theta}=\theta+\alpha,\;\;\bar{\omega}=\Phi_1\omega,
\end{align}
where $Z\in\RR^{2d}$, $\theta\in\TT^m$, $\omega\in\Omega$ are variables, $\alpha\in\RR^m$ is the rotation vector, and $\Phi_1$ represents the time-1 Wiener shift. Mathematically, the system in the form of \eqref{discrete-stochastic-system} is referred to as a \textit{random (quasi-periodic) skew-product}:
\begin{align}\label{skew-product} 
(F,\alpha,\Phi): \; \RR^{2d}\times\TT^m\times\Omega &\to \RR^{2d}\times\TT^m\times\Omega \notag\\
\;\;\;\;(Z,\theta,\omega)\;\;\;\;&\to\;\;\;\;(\bar{Z},\bar{\theta},\bar{\omega}),
\end{align}
where $\RR^{2d}\times\TT^m\times\Omega$ is understood as a trivial measurable vector bundle \cite{Arnold98}.

\begin{remark}
Clearly, for $(F,\alpha,\Phi)$, the space variable $Z$ depends on realizations of the noise, while the internal phase $\theta$ influences the motion of $Z$ but remains unaffected by it. Here, we focus on the discrete case, as it is more suitable for numerical implementations. Indeed, the proofs of our main results are designed to be readily implementable on a computer. Moreover, we remark that results for flows can often be derived from those for maps using standard techniques, such as the Poincar\'e trick \cite{HdlL06b}. However, we will not include these derivations in this paper.
\end{remark}

\begin{remark}\label{rk:deterministic}
An important special case occurs when the system takes the form $F(z,\theta,\omega)= F_0(z,\theta)+F_1(z,\theta,\omega)$, where $F_1$ is small (i.e., the random forcing is small). For $F_1\equiv 0$ (i.e., in the absence of random noise), the torus 
$\mathcal{K}_0=\{(K_0(\theta),\theta):\theta\in\TT^m \}
$ has been shown to be invariant under the deterministic skew-product system \cite{ZhangL18}, where $K_0:\TT^m\to\RR^{2d}$ is an embedding such that 
\begin{equation}
  \label{eq:deterministic}
  F_0(K_0(\theta),\theta)=K_0(\theta+\alpha).
\end{equation}
Motivated by this, in the presence of random perturbations, it is natural to look for random invariant manifolds that are close to $\mathcal{K}_0$ in the random systems.
\end{remark}

Let $K:\TT^m\times\Omega\to \RR^{2d}$ be a random injective immersion. It defines a random submanifold 
\begin{equation}
\mathcal{K}=\big\{\mathcal{K}_\theta(\omega)=\big(K(\theta,\omega),\theta\big):\theta\in\TT^m \big\}_{\omega\in\Omega}
\end{equation}
which is indeed a random torus, and we say that $\mathcal{K}$ is parameterized by $K$. According to Definition \ref{defn: RIM}, the random submanifold $\mathcal{K}$ is invariant if
\begin{equation}\label{random bundle invariant-0}
F\big(K(\theta,\omega),\theta,t,\omega\big)=K(\theta+t\alpha,\Phi_t\omega).
\end{equation}
In the special case with $t=1$, we obtain 
\begin{equation}\label{random bundle invariant}
F\big(K(\theta,\omega),\theta,\omega\big)=K(\theta+\alpha,\Phi_1\omega).
\end{equation}
Note that when $K$ is continuous, the equation \eqref{random bundle invariant} is equivalent to \eqref{random bundle invariant-0} due to the following ergodicity.

\begin{lem}
  The map
  $T_{\alpha,\Phi_1}(\theta,\omega)=(\theta+\alpha,\Phi_1\omega)$ is ergodic in
  $\TT^m\times\Omega$.
\end{lem}

\begin{proof}
  It suffices to show that the only $L^2$-invariant function over $\TT^m\times\Omega$ is constant. We
  consider the partial Fourier series $
  g(\theta,\omega)=\sum_j\hat{g}_{j}(\omega)e^{2\pi ij\theta}$ for any $g\in L^2(\TT^m\times\Omega)$. The function $g$ is invariant if and only if, for every $j\in\ZZ$, the following condition holds:
  \begin{equation}
    \label{eq:inv_func}
    \hat{g}_j(\omega)=e^{2\pi ij\alpha}\hat{g}_j(\Phi_1\omega).
  \end{equation}
  It is clear that when $j=0$, the equation \eqref{eq:inv_func} simplifies to $\hat{g}_0(\omega)=\hat{g}_0(\Phi_1\omega)$, which is constant by the ergodicity of the shift $\Phi$. Thus, we only need to discuss the case for $j\neq 0$. Notice that $\mathbb{E}(\hat{g}_j)=0$ when $j\neq 0$. Iterating the equation \eqref{eq:inv_func} and multiplying by the conjugate $\hat{g}^*_j$, we infer that, for any integer $n>0$, 
    $|\hat{g}_j|^2=e^{2\pi ij\alpha n}\hat{g}_j(\Phi_n\omega)\hat{g}^*_j(\omega).$ We can now take the expected value and select a sequence of $n$ such that $e^{2\pi ij\alpha n}\to 1.$ Applying the decay of correlations and the fact that the expected values of $g_j$ is zero, we obtain that the right-hand side can be made as small as desired. Hence $\mathbb{E}(|\hat{g}_j|^2)=0$. So these Fourier coefficients are zero.
\end{proof}

Similar to the deterministic case, both \eqref{random bundle invariant-0} and \eqref{random bundle invariant} are referred to as the \textit{random invariant equations}. These equations, along with their variants, will form the centerpiece of our analysis in the following sections.

\section{Existence and persistence of random invariant tori}\label{sec:3}

With the preliminary work outlined in Section \ref{Setups}, we convert the original SDE \eqref{Cauchy-SDE} into the equivalent RDE \eqref{eq: random}, and eventually reach a random skew-product system \eqref{discrete-stochastic-system}. The key task now is to search for invariant tori of \eqref{discrete-stochastic-system}, parameterized by a mapping that satisfies the random invariant equation \eqref{random bundle invariant}. It is important to note that \eqref{discrete-stochastic-system} is a very general system, and the results presented in this paper are applicable to a wide range of models.

\subsection{Spaces of differentiable and measurable functions}

In random settings, the dependence on $\omega$ needs to be merely measurable, not even continuous, since the sample space $\Omega$ is not assumed to be a metric space (let alone a manifold) \cite{Cheng2019}. To obtain sharp results on the regularity of the invariant manifolds, it is important to distinguish between the regularities of the functions with respect to the horizontal variables $\theta$ and the vertical variables $z$. The following spaces are introduced to facilitate induction arguments for the functional equations. They are an adaptation of the definitions used in the deterministic parameterization method \cite{CFdelaL2003a,CFdelaL2003b}.

\begin{definition}\label{fun-space}
    ($\mathcal{L}^\infty (C^\Sigma)$-spaces) Let $\Sigma \subset \mathbb{N}^2$ be a subsets such that: $(i^{\prime},j^{\prime})\in \Sigma$ if $(i,j)\in \Sigma$ and $i^{\prime}\leqslant i$, $j^{\prime}\leqslant j$.  We denote by 
\begin{align}\label{LC-space-r}
&\mathcal{L}^\infty (C^\Sigma)=\mathcal{L}^\infty ({\Omega}, C^\Sigma(\RR^{2d}\times \TT^m,\RR^{2d}) ) \notag\\ 
=&\big\{ F : \RR^{2d}\times\TT^m\times\Omega \to \RR^{2d}\;|\;F\in (\mathscr{B}(\RR^{2d})\otimes \mathscr{B}(\TT^m)\otimes \mathscr{F},\mathscr{B}(\RR^{2d})),\notag\\
 & \text{  $D_\theta^iD_z^jF(\cdot,\cdot,\omega)$ exists and is continuous and bounded for each $(i,j)\in \Sigma$}\notag\\
 & \text{  and a.e. $\omega\in {\Omega}$}\},
\end{align}
and define the norm on $\mathcal{L}^\infty (C^\Sigma)$ by 
\begin{align}
\|F\|_{\mathcal{L}^\infty (C^\Sigma)}
=&\sup_{i,j\in\NN,\;(i,j)\in \Sigma}|D_\theta^iD_z^jF(z,\theta,\omega)|_\infty \notag\\
=&\sup_{i,j\in\NN,\;(i,j)\in \Sigma}  \inf \big\{c\in\overline{\RR}:\;\PP\{\omega: |D_\theta^iD_z^jf(z,\theta,\omega)|>c\}=0\big\},\notag
\end{align}
where $|\cdot|_\infty$ is the essential supremum norm \cite{Ash2000}. In other words, the function $D_\theta^iD_z^jF(\cdot,\cdot,\omega)$ may have countable discontinuities and is bounded outside a set of measure 0. 
Then the induced metric on $\mathcal{L}^\infty (C^\Sigma)$ makes it into a Banach space

In particular, we say $F$ is $\mathcal{L}^\infty (C^{r,s})$, (jointly) $\mathcal{L}^\infty (C^{r})$, or $\mathcal{L}^\infty (C^{\Sigma_{r,s}})$, if $F\in \mathcal{L}^\infty ( C^\Sigma)$ with $\Sigma_1=\{(i,j)\in \mathbb{N}^2\; | \; i\leqslant r, j\leqslant s\},$ $\Sigma_2=\{(i,j)\in \mathbb{N}^2\; | \; i+j\leqslant r\},$ or $\Sigma_3=\Sigma_{r,s}=\{(i,j)\in \mathbb{N}^2\; | \; i\leqslant r, i+j\leqslant r+s\},$ respectively. It is clear that $\mathcal{L}^\infty (C^{r+s} )\subset \mathcal{L}^\infty (C^{\Sigma_{r,s}}) \subset \mathcal{L}^\infty (C^{r,s})$.
\end{definition}

\subsection{Transfer operators and the normal hyperbolicity}

 Let $\mathbf{\Gamma}$ be the space of functions $\kappa: \TT^m \times \Omega \to \mathbb{R}^{2d}$. It can also be understood as the space of random sections, owing to the triviality of the random bundle $\mathbb{R}^{2d} \times \TT^m \times \Omega$. We can then define the regularity of $\mathbf{\Gamma}$ in a manner similar to that in $\mathcal{L}^\infty (C^\Sigma)$. For example, we can consider random bounded sections (${\mathbf{\Gamma}}_b$), random continuous sections (${\mathbf{\Gamma}}_{C^0}$), $\mathcal{L}^\infty (C^r)$ sections (${\mathbf{\Gamma}}_{C^r}$), and so on.
 
We can now rewrite the random invariance equation \eqref{random bundle invariant} in a more functional parlance. Specifically, $K$ is a generalized (or random) fixed point of the graph transform functional $\mathscr{G}: \mathbf{\Gamma} \to \mathbf{\Gamma}$, defined by
 \begin{align}
\mathscr{G}(K)(\theta,\omega)=F\big(K(\theta-\alpha,\Phi_{-1}\omega ),\theta-\alpha,\Phi_{-1}\omega \big).
\end{align}
The linearized dynamics around the corresponding graph (i.e., $\mathcal{K}$) is described by the vector bundle map:
\begin{align}\label{random vector bundle map}
(M,\alpha,\Phi): \; \RR^{2d}\times\TT^m\times\Omega &\to \RR^{2d}\times\TT^m\times\Omega \notag\\
\;\;\;\;({\bf v},\theta,\omega)\;\;\;\;\;&\to\;(\bar{\bf v},\bar{\theta},\bar{\omega}):=\big(M(\theta,\omega){\bf v},\theta+\alpha,\Phi_1\omega\big),
\end{align}
where $M(\theta,\omega)=D_zF\big(K(\theta,\omega),\theta,\omega\big)$ is the \textit{(random) transfer matrix}. To encode the forward and backward dynamics of the linear skew-product, we adopt the following cocycle notation:
\begin{align}
M(\theta,\omega,N)=
\left\{
\begin{array}{rl}
M\big(\theta+(N-1)\alpha,\Phi_{N-1}\omega\big)\cdots M(\theta,\omega),\quad\quad&\text{ for } N>0,\\
\text{Id},\quad\quad\quad\quad\quad\quad\quad\quad\quad\quad\quad\quad\quad\quad\quad\quad\quad&\text{ for } N=0,\\
M\big(\theta+N\alpha,\Phi_{N}\omega\big)^{-1}\cdots M(\theta-\alpha,\Phi_{-1}\omega)^{-1},&\text{ for } N<0. \notag\\
\end{array}
\right.
\end{align} 
Here, for $N< 0$, the inverses of the matrices are assumed to be well-defined. The linear skew-product $(M,\alpha,\Phi)$ induces a transfer operator $\mathscr{M}:  {\bf\Gamma}  \to {\bf\Gamma} $, given by
 \begin{align}
\mathscr{M}[\kappa](\theta,\omega)=M\big(\theta-\alpha,\Phi_{-1}\omega \big)\kappa\big(\theta-\alpha,\Phi_{-1}\omega \big).
\end{align}
The operator norm of the transfer operator $\mathscr{M}$ is defined as
$$\|\mathscr{M}\|=\sup\{ \|\mathscr{M}\kappa\|_{\mathcal{L}^\infty (C^r)}: \kappa \in \Gamma_{C^r}, \|\kappa\|_{\mathcal{L}^\infty (C^r)}=1\}.$$

\begin{remark}\label{spectrum-remark}
Note that the transfer operator $\mathscr{M}$ is defined in the sense of $\omega$-wise, and it shares many of the same properties as in the deterministic case for each fixed $\omega\in\Omega$. 
Particularly necessary to point out that the spectrum of the transfer operator does not depend on the space it is considered to act on, since the motion on the base is a rotation here \cite{HarodlLlave2003,HdlL06b}. This means that we have $\text{Spec} (\mathscr{M},{\bf\Gamma}_b)=\text{Spec} (\mathscr{M},{\bf\Gamma}_{C^r})$ for all $\omega\in\Omega$. This fact can be leveraged to obtain certain results with strong regularity by formulating weaker spectral gap assumptions. 
\end{remark}

We have the following definition for normally hyperbolic random invariant tori, referring to \cite{Liji2013}.
\begin{definition} \label{defn:NHRIM}
(Normal hyperbolicity) A random invariant torus $\mathcal{K}$ of the random skew-product $(F,\alpha,\Phi)$, parameterized by $K\in {\bf\Gamma}_{C^r}$, is said to be ($\omega$-wise) normally hyperbolic, if the corresponding linear skew-product $(M,\alpha,\Phi)$ is normally hyperbolic. Specifically, if for all $\omega\in \Omega$ and $z\in \mathcal{K}(\omega)$, there exists a splitting which is $C^r$ in $z$ and measurable in $\omega$:
$$
\RR^{2d}\times\TT^m =\mathscr{E}^S(z,\omega)\oplus \mathscr{E}^U(z,\omega)
$$
of closed subspaces with associated projections $\Pi^U(z,\omega)$ and $\Pi^S(z,\omega)$, and constants $C(\omega)>0$ and $0<\lambda_S(\omega) <1<\lambda_U(\omega)$ such that:
\begin{enumerate}
\item[$\bullet$] ${\bf v}\in\mathscr{E}^S$ if and only if $|M(\theta,\omega,N){\bf v}|\leqslant C \lambda_S^N(\omega)|v|$ for all $N\geqslant 0$;
\item[$\bullet$] ${\bf v}\in\mathscr{E}^U$ if and only if $|M(\theta,\omega,N){\bf v}|\leqslant C \lambda_U^N(\omega)|v|$ for all $N\leqslant 0$.
\end{enumerate}
In particular, the stable bundle $\mathscr{E}^S$ and the unstable bundle $\mathscr{E}^U$ are invariant under the action of the linear skew-product $(M,\alpha,\Phi)$, defining linear skew-products $(\Lambda_S,\alpha,\Phi)$ and $(\Lambda_U,\alpha,\Phi)$, respectively. In particular, $(\Lambda_U,\alpha,\Phi)$ is invertible.
\end{definition}

\begin{remark}\label{hyperbolicity-remark}
Note that the $\omega$-wise normal hyperbolicity is equivalent to the condition that, for each $\omega\in \Omega$, the spectrum of the transfer operator $\mathscr{M}$ does not intersect the unit circle $\{ y\in\CC:|y|=1\}$ \cite{Mane74}. This means that for every $y$ such that $|y|=1$, there exists a positive constant $c_H$ (known as the hyperbolicity bound) such that $\|(\mathscr{M}-y \text{Id})^{-1}\|\leqslant c_H$.
\end{remark}



\subsection{The main theorem for random invariant tori}

We now present the main theorem on the existence and persistence of random invariant tori. This theorem is mainly inspired by the deterministic work of Haro and de la Llave \cite{HdlL06b}, with the goal of generalizing and enhancing their results to the random setting. We highlignt that, for a fixed $\omega\in\Omega$, many of the arguments in this theorem can be reduced to the deterministic case.  Particular attention, however, is required for the terms involving the Wiener shift and the measurability with respect to $\omega$.

\begin{theorem}\label{Existence-Persistence}
Consider a random skew-product $(F,\alpha,\Phi)$ as given in\eqref{skew-product}. Let $U \subset \mathbb{R}^{2d}$ be an open set and $r \geqslant 0$. We assume that $F\in \mathcal{L}^\infty (\Omega,C^{\Sigma_{r,1}}(U\times \TT^m, \RR^{2d})$, such that, for each fixed $\theta\in\TT^m$ and $\omega\in\Omega$, $F(\cdot,\theta,\omega)$ is a local diffeomorphism.

Let $K_\ast:\TT^m\times\Omega\to U\subset \RR^{2d}$ be a random map in $\mathcal{L}^\infty (\Omega,C^r(\TT^m,\RR^{2d}))$  such that:
\begin{enumerate}
\item[$(i)$] $K_*$ parameterize an approximate random invariant torus, that is
\begin{align}
\| F\big(K_\ast(\theta,\omega),\theta,\omega\big)-K_\ast(\theta+\alpha,\Phi_1\omega) \|_{\mathcal{L}^\infty (C^r)}<\epsilon.
\end{align}
\item[$(ii)$] For the $C^r$ matrix-valued random map $M_\ast:\TT^m\times\Omega\to\text{GL}(\RR^{2d})$, defined by 
$
M_\ast(\theta,\omega)=D_zF\big(K_\ast(\theta,\omega),\theta,\omega\big),
$ 
 the corresponding transfer operator $\mathscr{M}_\ast$ satisfies the hyperbolic condition:
\begin{align}\label{hyperbolic condition}
\|(\mathscr{M}_\ast-y \text{Id})^{-1}\|\leqslant c_H, 
\end{align}
for each $y\in\mathbb{C}$ with $|y|=1$, where $c_H>0$ is the hyperbolicity bound.
\end{enumerate}
Then, the following results hold:
\begin{enumerate}
\item[$(a)$] (Existence) If $\epsilon>0$ is sufficiently small, there exists a $C^r$ random map $K_\epsilon:\TT^m\times\Omega \to U\subset \RR^{2d}$ such that
\begin{align}\label{invariant eqn-thm1}
F\big(K_\epsilon(\theta,\omega),\theta,\omega\big)=K_\epsilon(\theta+\alpha,\Phi_1\omega),
\end{align}
and $\|K_\epsilon-K_\ast \|_{\mathcal{L}^\infty (C^r)}=\mathcal{O}(\epsilon)$.
\item[$(b)$] (Uniqueness) In a $\mathcal{L}^\infty (C^0)$ neighborhood of $K_\ast$, the solution $K_\epsilon$ above is the unique $\mathcal{L}^\infty (C^0)$ solution of \eqref{invariant eqn-thm1}.

\item[$(c)$] (Persistence) The map $F \to K_\epsilon $ is $\mathcal{L}^\infty (C^1)$ when $F$ is given the $\mathcal{L}^\infty (C^{\Sigma_{r,1}})$ topology and $K_\epsilon$ the $\mathcal{L}^\infty (C_r)$ topology.

\item[$(d)$] (Normal hyperbolicity) The random torus $\mathcal{K}_\epsilon$ parameterized by $K_\epsilon$ is normally hyperbolic.

\end{enumerate}
\end{theorem}

\begin{proof} The key to the proof lies in solving the random invariant equation \eqref{invariant eqn-thm1}. To this end, we define a operator $\mathscr{T}_F: \mathcal{L}^\infty (\Omega,C^r(\TT^m,U))\to \mathcal{L}^\infty (\Omega,C^r(\TT^m,\RR^{2d}))$ by
\begin{align}
\mathscr{T}_F(K)(\theta,\omega)=F\big(K(\theta-\alpha,\Phi_{-1}\omega),\theta-\alpha,\omega\big)-K_\epsilon(\theta,\omega).
\end{align}
It is clear that, for each $\omega\in\Omega$, $\mathscr{T}_F$ is a $C^1$ operator since $F\in\mathcal{L}^\infty(C^{\Sigma_{r,1}})$. Its derivative is given by
\begin{align*}
  D\mathscr{T}_F(K_)\Delta(\theta,\omega)=&D_zF\big(K(\theta-\alpha,\Phi_{-1}\omega),\theta-\alpha,\Phi_{-1}\omega\big)\Delta(\theta-\alpha,\Phi_{-1}\omega)\\
&-\Delta(\theta,\omega),
\end{align*}
which simplifies to  $D\mathscr{T}_F(K)=\mathscr{M}-Id$ with $\mathscr{M}$ being the transfer operator associated to $M(\theta,\omega)=D_zF\big(K(\theta,\omega),\theta,\omega\big)$. The primary task is thus to demonstrate that $D\mathscr{T}_F(K)$ is invertible as a linear operator acting on $C^r$ measurable sections $\Delta$. Upon establishing this, the existence and uniqueness of $K=K_\epsilon$ in $\mathcal{L}^\infty({C^r})$ spaces follows directly from the Inverse Function Theorem, which preserves the measurability.

We begin by asserting that the exact transfer operator $\mathscr{M}_\epsilon$ is hyperbolic and close to the approximate one $\mathscr{M}_\ast$. On the one hand, keeping in mind that $F$ is $\mathcal{L}^\infty(C^{\Sigma_{r,1}})$ and applying the assumption $(i)$, we deduce that 
\begin{align}\label{step1-thm1}
    \|{M}_\epsilon-{M}_\ast\|_{\mathcal{L}^\infty (C^0)}\leqslant \rho\big(\|K_\epsilon-K_\ast \|_{\mathcal{L}^\infty (C^0)}\big)\leqslant \rho(C\epsilon),
\end{align}
where $\rho$ denotes the modulus of continuity of $D_zF$ and $C$ is positive constant. We remark that, if $F$ is $\mathcal{L}^\infty(C^{\Sigma_{r,2}})$, the second inequality in \eqref{step1-thm1} can be replaced by $\|F\|_{\mathcal{L}^\infty(C^{\Sigma_{r,2}})} \|K_\epsilon-K_\ast \|_{\mathcal{L}^\infty (C^0)}$ in accordance with the mean value theorem. On the other hand, based on the assumption $(ii)$, we write
\begin{align}\label{M-invertible}
\mathscr{M}_\epsilon-y \text{Id}=(\mathscr{M}_\ast-y \text{Id})\left[\text{Id}+(\mathscr{M}_\ast-y \text{Id})^{-1}(\mathscr{M}_\epsilon-\mathscr{M}_\ast)\right],
\end{align}
for each $y\in\mathbb{C}$ with $|y|=1$. It is clear that
$$
\left\|(\mathscr{M}_\ast-y \text{Id})^{-1}(\mathscr{M}_\epsilon-\mathscr{M}_\ast)\right\|\leqslant c_H \rho(C\epsilon)<1,
$$
for sufficiently small $\epsilon>0$. By standard Neumann series arguments, we infer that the right-hand side of \eqref{M-invertible} is invertible. Thus, the transfer operator $\mathscr{M}_\epsilon$ is hyperbolic, and $D\mathscr{T}_F(K_\epsilon)$ is invertible. As mentioned in Remark \ref{spectrum-remark}, the spectrum over bounded sections is the same as that over $C^r$ sections. Hence, we have the result with $C^r$ regularity in the conclusion (a). In additional, the uniqueness in the conclusion (b) follows immediately by the Inverse Function Theorem in $C^0$ space. Let $\mathcal{K}_\epsilon$ be the torus parameterized by $K_\epsilon$.  It remains to prove that $\mathcal{K}_\epsilon$ is a random set, that is, we need to verify that, for any $z\in \RR^{2d}$
\begin{equation}
  \label{Thm1:measurable}
\omega\mapsto \inf_{\theta\in  \TT^d} |z-K_\epsilon(\theta,\omega)  |   
\end{equation}
is measurable. This can be done by the following standard procedure, referring to Theorem III.9 in \cite{Castaing1977}; see also Lemma 5.3 in \cite{DuanLS03}. 
Let $A_c$ be a countable dense set of $ \TT^m$. By the continuity of map $K_\epsilon(\cdot,\omega)$, the above infimum is equal to 
\begin{equation}
  \label{Thm1:measurable2}
\inf_{\theta\in A_c}  |z-K_\epsilon(\theta,\omega)  |.
\end{equation}
Since $\omega\mapsto K_\epsilon(\theta,\omega)$ is measurable for all $\theta\in \TT^m$, the mensurability of (\ref{Thm1:measurable2}) follows, proving the conclusion (d).
 
To establish conclusion (c), we further define a operator $\mathscr{T}:\mathcal{L}^\infty(\Omega,C^{\Sigma_{r,1}}(U\times\TT^m,\RR^{2d})) \times \mathcal{L}^\infty(\Omega,C^r(\TT^m,U))\to\mathcal{L}^\infty(\Omega,C^r(\TT^m,\RR^{2d}))$ by
\begin{equation}
  \label{Thm1:C1-operator}
\mathscr{T}(F,K)(\theta,\omega)=F\big(K(\theta-\alpha,\Phi_{-1}\omega),\theta-\alpha,\Phi_{-1}\omega\big)-K(\theta,\omega).
\end{equation}
This operator is $C^1$ for fixed $\omega\in\Omega$. The persistence of the random torus under perturbations and the $\mathcal{L}^\infty (C^1)$ dependence on $F$ follows by applying the Implicit Function Theorem. Thus we are done with the proof of Theorem \ref{Existence-Persistence}.
\end{proof}


Theorem \ref{Existence-Persistence} shows that if the error bound $\epsilon$ in hypothesis (i) is sufficiently small and the hyperbolicity condition in hypothesis (ii) holds, then there exists a nearby random invariant torus, which is also hyperbolic. In applications, particularly in the case of small random noise, the approximate invariant torus is often considered to be deterministic (i.e., $K_\ast(\theta,\omega)=K_0(\theta)$ which is independent of $\omega$). For this scenario, hypothesis (i) can be replaced by the condition:
\begin{equation}
  \label{F-F0}
\|F(z,\theta,\omega)-F_0(z,\theta) \|_{\mathcal{L}^\infty (C^r)}=\mathcal{O}(\epsilon),
\end{equation}
where $F_0$ satisfies the deterministic invariance equation \eqref{eq:deterministic}. The matrix-valued random map $M_\ast$ and its transfer operator $\mathscr{M}_\ast$ in hypothesis (ii) reduce to their deterministic counterparts: $M_0(\theta)=D_zF\big(K_0(\theta),\theta\big)$, $\mathscr{M}_0(\kappa)(\theta)=M_0(\theta-\alpha)\kappa(\theta)$, for $\kappa\in C^r(\TT^d,\RR^n)$. We summarize the corresponding results as a direct corollary of Theorem \ref{Existence-Persistence}.

\begin{corollary}\label{Cor:K0}
Consider the random skew-product $(F,\alpha,\Phi)$ given in Theorem \ref{Existence-Persistence}. Let $(F_0,\alpha)$ be the corresponding deterministic skew-product satisfies \eqref{F-F0}. Assume that we are given a normally hyperbolic (deterministic) invariant torus $\mathcal{K}_0$ parameterized by $K_0$, where $K_0$ is a $C^r(\TT^m,U)$ map solving the invariance equation \eqref{eq:deterministic}. Then, for sufficiently small $\varepsilon$, there is a unique normally hyperbolic random invariant torus $\mathcal{K}_\varepsilon$ being close to $\mathcal{K}_0$ in the sense of \eqref{invariant eqn-thm1}.
\end{corollary}


It is pointed out that Theorem \ref{Existence-Persistence} permits us to guarantee the existence of $\mathcal{L}^\infty(C^r)$ random invariant torus near the $\mathcal{L}^\infty(C^r)$ approximate one. In fact, we can also deduce the existence of a $\mathcal{L}^\infty(C^r)$ random invariant torus from the existence of a $\mathcal{L}^\infty(C^0)$ approximate invariant torus. Based on Theorem \ref{Existence-Persistence}, we only need to claim that a $\mathcal{L}^\infty(C^0)$ random invariant torus of a $\mathcal{L}^\infty(C^r)$ random quasi-periodic skew-product, is necessarily $\mathcal{L}^\infty(C^r)$.

\begin{theorem}\label{Cor:boostsrap}
Consider the random skew-product $(F,\alpha,\Phi)$ given by \eqref{skew-product} with $F\in \mathcal{L}^\infty (\Omega,C^{r}(U\times \TT^m, \RR^{2d})$ being a map such that $F(\cdot,\theta,\omega)$ is a local diffeomorphism for all $\theta\in\TT^m$ and $\omega\in\Omega$, where $U\subset \RR^{2d}$ is an open set and $r\geqslant 0$. Let $K_\epsilon\in \mathcal{L}^\infty( \Omega,C^0( \TT^m, U))$ be a parameterization of the normally hyperbolic random invariant torus $\mathcal{K}_\epsilon$. Then, the parameterization $K_\epsilon$ is indeed $\mathcal{L}^\infty(C^r)$.
\end{theorem}
\begin{proof}
Notice that such a bootstrap on the regularity (with respect to $\theta$) neither depends on nor alters the measurability. In $\omega$-wise sense, this corollary can be proven using a technique similar to that in the proof of Theorem 3.9 in \cite{HdlL06b}; a related technique also appears in \cite{LlaveWayne1995}. Here, we only outline the main idea: Firstly, we show that the formal equations for the derivatives have unique solutions, which are continuous. Then, under the regularity assumptions for $F$, we demonstrate that the Taylor expansions obtained from these derivatives satisfy the equations, with a smallness condition that is proportional to a power of the displacement. Finally, using the quantitative estimates from Theorem \ref{Existence-Persistence}, we conclude that the $\mathcal{L}^\infty(C^0)$ solution $K_\epsilon$ differs from its Taylor approximation by less than a power, and by the converse of Taylor’s theorem, we deduce that $K_\epsilon$ is indeed $\mathcal{L}^\infty(C^r)$.
\end{proof}

\section{Perturbation analysis of random invariant tori}

In this section, we turn to studying the asymptotic expressions of random invariant tori based on the approach of perturbation theory (see, e.g., \cite{FW2012}). Observe that, in both the SDE \eqref{Cauchy-SDE} and the RDE \eqref{eq: random}, the randomness is controlled by the small parameter $\varepsilon$. The random terms are $\varepsilon\sigma \dot{W}_t$ and $\varepsilon\zeta(\Phi_t \omega)=\varepsilon\int_{-\infty}^0e^{-As}\sigma\Phi_t\omega(s)ds$, respectively. 

Given an initial value $z_0$. To emphasize the dependence on $\varepsilon$ and to apply our previous results, we rewrite the solution of RDE \eqref{eq: random} as $Z(z_0,t,\varepsilon)$, and denote by $F_\varepsilon(z_0,\theta,t,\varepsilon,\omega)$ the corresponding extended random flow. According to Lemma \ref{RDE}, the solution of  the SDE \eqref{Cauchy-SDE} is thus $z(z_0,t,\varepsilon)=T^{-1}(Z(z_0,t,\varepsilon),\Phi_t\omega)$, and the extended random flow satisfies  $\widehat{F}_\varepsilon:=T^{-1}(\cdot,\Phi_t\omega)\circ F_\varepsilon(T(z,\omega),\theta,t,\varepsilon,\omega)$.

\subsection{Expansion formulas for $F_\varepsilon$ and $K_\varepsilon$}\label{sec:perturbation theory}

The main idea of the perturbation theory is, provided that the functions $F_\varepsilon$ are differentiable in some appropriate function spaces (which we will discuss later in this section), we can consider an expansion
\begin{align}
\label{expansion}
K_\varepsilon=K_0+\varepsilon K_1+\cdots+\varepsilon^k K_k+\cdots
\end{align}
of $K_\varepsilon$ in powers of $\varepsilon$. By substituting this expansion with unknown coefficients $K_0,K_1,\cdots,K_k,\cdots$ into the random invariance equation
\begin{equation}\label{random bundle invariant-2}
F_\varepsilon\big(K_\varepsilon(\theta,\omega),\theta,\varepsilon,\omega\big)=K_\varepsilon(\theta+\alpha,\Phi_1\omega),
\end{equation}
we can expand both sides in powers of $\varepsilon$, where $F_\varepsilon(z,\theta,\varepsilon,\omega)=F_\varepsilon(z,\theta,1,\varepsilon,\omega)$ is the time-1 map associated with the solution of RDE \eqref{eq: random}.

To this end, we need to discuss how the left side of \eqref{random bundle invariant-2} is expanded in powers of $\varepsilon$. By regarding $K_\varepsilon=K(\varepsilon)$ as a power series with coefficients given in \eqref{expansion}, we write
\begin{align}\label{F-k}
    F_k=F_k(K_0,K_1,\cdots,K_k,\theta,\omega)=\frac{1}{k!}\frac{d^kF_\varepsilon(K_\varepsilon,\theta,\varepsilon,\omega)}{d\varepsilon^k}\bigg|_{\varepsilon=0}.
\end{align}
In particular, $F_0=F_\varepsilon(K_0,\theta,0)$, for $k=0$. As shown in Proposition \ref{pro-expansion}, we can see that, for $k\geqslant 1$, $F_k$ depends linearly on $K_k$, and the remainder $\mathcal{R}_{k-1}=\mathcal{R}_{k-1}(K_0,\cdots K_{k-1})(\theta,\omega)=F_k-M_0(\theta)K_k$ is independent of $K_k$, where $M_0(\theta)=D_zF_0( K_0,\theta)=D_zF_\varepsilon( K_0,\theta,0,\omega)$ is indeed the deterministic transfer matrix. 

Equating the coefficients of the same powers on left and right of \eqref{random bundle invariant-2}, we obtain equations for the successive calculation of the coefficients $K_1,\cdots,K_k,\cdots$ in \eqref{expansion}:
\begin{align}
\label{expansion-m}
K_0(\theta+\alpha)&=F_0(K_0,\theta),    \notag\\
K_1(\theta+\alpha,\Phi_1\omega)&=F_1(K_0,K_1,\theta,\omega)=M_0(\theta)K_1(\theta,\omega)+D_\varepsilon F_\varepsilon(K_0,\theta,0,\omega), \notag\\
\vdots\;\;\;&\;\;\;\;\;\;\vdots\\
K_k(\theta+\alpha,\Phi_1\omega)&=F_k(K_0,\cdots,K_k,\theta,\omega)\notag\\
&=M_0(\theta)K_k(\theta,\omega)+\mathcal{R}_{k-1}(K_0,\cdots K_{k-1})(\theta,\omega).\notag\\
\vdots\;\;\;&\;\;\;\;\;\;\vdots\notag
\end{align}
Note that these equations determine the functions $K_1,\cdots,K_k,\cdots$ uniquely, if $F_\varepsilon$ is sufficiently smooth. The zeroth approximation is determined from the first equation of the system \eqref{expansion-m}, which coincides with \eqref{eq:deterministic}. If $K_0$ is known, then the second equation in \eqref{expansion-m} is a linear equation in $K_1$. In general, if the functions $K_1,\cdots,K_{k-1}$ are known, then the equation for $K_k$ will be a linear equation.

\begin{pro}\label{pro-expansion}
(i) For $k\geqslant 1$, 
    \begin{align}\label{F-k-2}
    F_k=
    &\frac{1}{k!}\bigg(
    D_zF_\varepsilon(K_0,\theta,0,\omega)\frac{\partial^k K_\varepsilon}{\partial \varepsilon^k }+\sum_{r=1}^{k}D_z^rF_\varepsilon(K_0,\theta,0,\omega)\mathbf{B}_{k,r}\bigg(\frac{\partial K_\varepsilon}{\partial \varepsilon },\frac{\partial^2 K_\varepsilon}{\partial \varepsilon^2 },\cdots,\frac{\partial^{k-1}K_\varepsilon}{\partial \varepsilon^{k-1} }\bigg)\notag\\
    &+D_\varepsilon^kF_\varepsilon(K_0,\theta,0,\omega)
    \bigg) 
    \triangleq 
    D_zF_\varepsilon(K_0,\theta,0,\omega)K_k+\mathcal{R}_{k-1}(K_0,\cdots K_{k-1})(\theta,\omega),
\end{align} 
where $\mathbf{B}_{k,r}$ is the Bell polynomial, that is, 
$$\mathbf{B}_{k,r}\big(f^{(1)},f^{(2)},\cdots,f^{(k-1)}\big)=\sum \frac{k!}{j_1!j_2!\cdots j_{k-1}!}\bigg(\frac{f^{(1)}}{1!}\bigg)^{j_1}\bigg(\frac{f^{(2)}}{2!}\bigg)^{j_2}\cdots \bigg(\frac{f^{(k-1)}}{k-1!}\bigg)^{j_{k-1}},
$$
for any $f\in C^k$, $j_1+j_2+\cdots+j_{k-1}=r \;\&\; j_1+2j_2+\cdots+(k-1)j_{k-1}=k$, and $D_z$, $D_\varepsilon$ stand for the derivatives with respect to the first and third variables respectively. The randomness of $F_\varepsilon$ only appears in the term $D^k_\varepsilon F_\varepsilon$.\\
(ii) Denote by $Z(z_0,t,\varepsilon)$ a solution to the RDE \eqref{eq: random}. Then $\frac{\partial^k Z}{\partial \varepsilon^k}(t,0)=\frac{\partial^k Z}{\partial \varepsilon^k}(t,\varepsilon)\big|_{\varepsilon=0}$ satisfies the following RDE:
\begin{align}\label{F-RDE-k}
    \frac{d}{dt}\frac{\partial^k Z}{\partial \varepsilon^k}\Big|_{\varepsilon=0}=&\big(A +D_z B(Z(t,0),t)\big)\cdot\frac{\partial^k Z}{\partial\varepsilon^k} \Big|_{\varepsilon=0}\\
    &+\sum_{r=1}^kD_z^rB(Z(s,0),s)\mathbf{B}_{k,r}\bigg(\frac{\partial Z}{\partial \varepsilon}+\zeta(\Phi_t\omega,t),\frac{\partial^2 Z}{\partial \varepsilon^2},\cdots,\frac{\partial^{k-1} Z}{\partial \varepsilon^{k-1}}\bigg)\Big|_{\varepsilon=0}.\notag
\end{align}
By introducing $\theta=\alpha t$ to \eqref{F-RDE-k} and taking the initial value $z_0=K_0$, we conclude that $D^k_\varepsilon F_\varepsilon$ in this case is given by the extended random flow of $\frac{\partial^k Z}{\partial \varepsilon^k}(t,0)$. In particularly, 
\begin{equation}\label{Z-epsilon-0}
    \frac{\partial Z}{\partial \varepsilon}(t,0)=e^{\int_0^t(A+D_zB(Z,s))ds}\int_0^te^{-\int_0^s(A+D_zB(Z,\tau))d\tau}\cdot
DB(Z,s)\cdot \zeta(\Phi_s\omega)ds,
\end{equation}
and
\begin{align}\label{R-0}
    \mathcal{R}_0(K_0)(\theta,\omega)=D_\varepsilon F_\varepsilon(K_0,\theta,0,\omega)=\frac{\partial Z}{\partial \varepsilon}(K_0,\theta,1,0,\omega).
\end{align}
\end{pro}
\begin{proof}
    The conclusion (i) follows immediately by Taylor's theorem and simple calculations. To prove the conclusion (ii), we note that the the solution of the RDE \eqref{eq: random} is interpreted as the integral equation \eqref{eq: Z}. By taking derivative with respect to $\varepsilon$, we infer that
\begin{align}
    \frac{\partial Z}{\partial \varepsilon}=\int_0^t A\cdot \frac{\partial Z}{\partial \varepsilon}(s,\varepsilon)+D_zB\big(Z(s,\varepsilon)+\varepsilon\zeta(\Phi_t\omega,s) ,s\big)\cdot \left(\frac{\partial Z}{\partial \varepsilon}(s,\varepsilon)+\zeta(\Phi_t\omega)\right)ds.\notag
\end{align}
Hence, by letting $\varepsilon=0$, we can find that $\frac{\partial Z}{\partial \varepsilon}(t,0)$ satisfies the linear RDE:
\begin{equation}\label{F-RDE}
    \frac{d}{dt}\frac{\partial Z}{\partial \varepsilon}(t,0)= A\cdot \frac{\partial Z}{\partial \varepsilon}(t,0)+D_zB\big(Z(t,0),t\big)\cdot \left(\frac{\partial Z}{\partial \varepsilon}(t,0)+\zeta(\Phi_t\omega)\right),
\end{equation}
which can be explicitly solved as given in \eqref{Z-epsilon-0}. Analogously, we can solve successively the RDEs for $k\geqslant 2$ to obtain $D^k_\varepsilon F_\varepsilon$. 
\end{proof}

\subsection{A perturbation theory for random invariant tori}

Based on the analysis in Section \ref{sec:perturbation theory}, we modify Definition \ref{fun-space} and consider the following function space.
\begin{definition} ($\mathcal{L}^\infty (C^{r,s,l})$-Space) 
     Let $i,j,k\in\NN$. We denote by
     \begin{align}
         &\mathcal{L}^\infty (C^{r,s,l})=\mathcal{L}^\infty ({\Omega},C^{\Xi}(\RR^{2d}\times \TT^m\times \RR^+\times\Omega,\RR^{2d}))  \notag\\ 
         =&\big\{ F_\varepsilon:\RR^{2d}\times \TT^m\times \RR^+\times \Omega \to \RR^{2d}\; | \;F_\varepsilon\in (\mathscr{B}(\RR^{2d})\otimes \mathscr{B}(\TT^m)\otimes \mathscr{F},\mathscr{B}(\RR^{2d})) \text{ for fixed $\varepsilon$,}\notag\\
         &\;\;\;D_\theta^iD_z^jD_\varepsilon^k F_\varepsilon \text{ exists and is continuous and bounded for each $i\leqslant r$, $j\leqslant s$, $k\leqslant l$,}\notag\\
         &\;\;\;\text{and a.e. $\omega\in {\Omega}$
         } \big\}, \notag
     \end{align}
     and define the norm on $\mathcal{L}^\infty (C^{r,s,l})$ by 
     \begin{align}
     \|F\|_{\mathcal{L}^\infty (C^{r,s,l})}=\sup_{i\leqslant r,j\leqslant s,k\leqslant l}|D_\theta^iD_z^jD_\varepsilon^k F_\varepsilon(z,\theta,\varepsilon,\omega)|_\infty. \notag
     \end{align}
\end{definition}
We formulate the following theorem concerning the expansions of random invariant tori in powers of a small parameter $\varepsilon$.

\begin{theorem}\label{Thm2:expansion}
Let $U\subset \RR^{2d}$ be an open set. Let $F_\varepsilon:U\times\TT^m\times\RR\times \Omega \to \RR^{2d}$ be a map of class $\mathcal{L}^\infty (C^{0,l,l})$ with $l\geqslant 1$, such that $(F_\varepsilon,\alpha,\Phi)$ forms a random skew-product (over the rotation $\alpha\in\RR^d$):
$$
\bar{Z}={F}_\varepsilon({Z},\theta,\varepsilon,\omega),\;\;\bar{\theta}=\theta+\alpha,\;\;\bar{\omega}=\Phi_1\omega,
$$
\begin{enumerate}
\item[$(1)$] Assume that we are given a normally hyperbolic deterministic invariant torus $\mathcal{K}_0$ parameterized by $K_0$, where $K_0$ is a $C^0(\TT^m,U)$ map solving the invariance equation \eqref{eq:deterministic}. Then, for sufficiently small $\varepsilon$, there exists a $\mathcal{L}^\infty(C^0)$ random map $K_\varepsilon:\TT^m\times\Omega \to U\subset \RR^{2d}$ such that the random invariance equation \eqref{random bundle invariant-2} holds and 
\begin{align}
\label{expansion-K}
K_\varepsilon (\theta,\omega)=K_0(\theta)+\varepsilon K_1 (\theta,\omega)+\cdots+\varepsilon^{l-1} K_{l-1} (\theta,\omega)+\mathcal{O}(\varepsilon^{l}),
\end{align}
where $K_1,\cdots,K_{l-1}\in\mathcal{L}^\infty(\Omega,C^0(\TT^d,U))$ are determined from the system \eqref{expansion-m}.

\item[$(2)$] If we further assume that $F_\varepsilon\in \mathcal{L}^\infty (C^{r,l,l})$, then $K_\varepsilon \in\mathcal{L}^\infty(C^r)$.

\end{enumerate}
\end{theorem}

\begin{proof} 
Note that, for fixed $\varepsilon$, the map $F_\varepsilon$ is reduced to the one ($F$) in Section \ref{sec:3}. Once we have proven the conclusion (1), the conclusion (2) follows immediately by Theorem \ref{Cor:boostsrap}. Hence we only prove the conclusion (1) here. Keeping the equations \eqref{expansion-m} in mind, the main idea of our proof is to apply the method of induction in the order of expansion. 

For $l=1$, the results follow directly by Corollary \ref{Cor:K0} with $r=0$. Later in Remark \ref{another-proof}, we will sketch an alternative proof which is based on a fixed point argument and attributes of the function $F_\varepsilon$ (instead of that of $F$).

Assume that the $(l-1)$th-order expansion of $K_\varepsilon$ exists, that is, $K_0,K_1,\cdots K_{l-1}$ are known. Formally expanding $F$ and $K_\varepsilon$ up to order $l$, we can see that the random invariance equation \eqref{random bundle invariant-2} at order $\varepsilon^{l}$ becomes
\begin{align}\label{expansion-linear}
M_0(\theta)K_l(\theta,\omega)-K_l(\theta+\alpha,\Phi_1\omega)=-\mathcal{R}_{l-1}(K_0,\cdots K_{l-1})(\theta,\omega),
\end{align}
where $M_0(\theta)=D_zF_\varepsilon( K_0(\theta),\theta)$ is the transfer matrix and $\mathcal{R}_{l-1}(\theta,\omega)$ is a known function given by \eqref{F-k}. Note that \eqref{expansion-linear} is a linear equation with deterministic coefficients, and the only random terms are in the right-hand side (and, of course, the unknowns).

A crucial point is the hypothesis of hyperbolicity and it means that the left-hand side of \eqref{expansion-linear} forms an invertible operator from $\mathcal{L}^\infty( \Omega,C^0( \TT^m, \RR^{2d}))$ to itself:
$$
(\mathscr{M}_0- \text{Id})[K_l](\theta+\alpha,\Phi_1\omega)=M_0(\theta)K_l(\theta,\omega)-K_l(\theta+\alpha,\Phi_1\omega),
$$
where $\mathscr{M}_0$ is the transfer operator associated to $M_0$. Therefore, by the Inverse Function Theorem, we conclude that there is a unique solution 
$$
K_l(\theta,\omega)=(\mathscr{M}_0- \text{Id})^{-1}[\mathcal{R}_{l-1}](\theta-\alpha,\Phi_{-1}\omega)
$$
of equation \eqref{expansion-linear} in $\mathcal{L}^\infty( \Omega,C^0( \TT^m, \RR^{2d}))$. Note that the random invariance equation \eqref{random bundle invariant-2} holds with $F_\varepsilon\in \mathcal{L}^\infty (C^{\Xi_{0,l,l}})$. There exists a positive constant $C$ such that 
$$
\Big\|K_\varepsilon-\sum_{k=0}^{l}\varepsilon^k K_k\Big\|_{\mathcal{L}^\infty(C^0)}\leqslant \Big\|F_\varepsilon-\sum_{k=0}^{l}\varepsilon^k F_k\Big\|_{\mathcal{L}^\infty(C^{\Xi_{0,l,l}})}\leqslant C\varepsilon^l.
$$
By induction, we are done with the proof.
\end{proof}

\begin{remark}\label{another-proof}
    Note that, comparing with the general setups for $F$ in Section \ref{sec:3}, the randomness in $F_\varepsilon$ is controlled by the small parameter $\varepsilon$ and the random terms can be calculated in this case. It is thus possible to discuss about the existence of the random invariant torus in a more specific way (e.g., estimating the value range of $\varepsilon$, etc). We now present a direct proof for the case $l=1$ by the following lines.

 For $l=1$, we are indeed looking for a solution $K_\varepsilon(\omega)$ of the random invariance equation \eqref{random bundle invariant-2} of the form 
$$
K_\varepsilon(\theta,\omega)=K_0(\theta)+\Delta_\varepsilon(\theta,\omega).
$$
The new unknown is $\Delta_\varepsilon \in \mathcal{L}^\infty(C^0)$, which should satisfy $\|\Delta_\varepsilon\|_{\mathcal{L}^\infty(C^0)}\leqslant C\varepsilon$ for some positive constant $C$. Let us define a closed set
$$
\mathscr{B}_{{r}}=\{\Delta\in L^\infty(\Omega,C^0(\TT^d,U))\;|\; \|\Delta\|\leqslant {r}\},
$$
which forms a complete metric space with respect to the restricted  $\mathcal{L}^\infty(C^0)$-norm. What we need to do is thus find the new unknown $\Delta_\varepsilon\in\mathscr{B}_{r}$ with a suitable $r$ that we will choose later. 
By Taylor’s theorem, the random invariance equation \eqref{random bundle invariant-2} is equivalent to 
 \begin{align}
 \label{M-random-N}
M_0(\theta-\alpha)\Delta_\varepsilon(\theta-\alpha,\Phi_{-1}\omega)-\Delta_\varepsilon(\theta,\omega)=\mathscr{N}(\Delta_\varepsilon)(\theta,\omega),
 \end{align}
 where $\mathscr{N}:\mathscr{B}_{r}\to \mathcal{L}^\infty(\Omega,C^0(\TT^m,\RR^{2d}))$ is the operator defined by the identity
 \begin{align}
\mathscr{N}(\Delta_\varepsilon)(\theta+\alpha,\Phi_1\omega) =&K_0(\theta+\alpha)-F_0(K_0,\theta)-\varepsilon \int_0^1 D_\varepsilon F_\varepsilon(K_0+\alpha\Delta_\varepsilon,\theta,\alpha\varepsilon,\omega) d\alpha \notag\\
&-\Delta_\varepsilon\int_0^1 \big[D_z F_\varepsilon(K_0+\alpha\Delta_\varepsilon,\theta,\alpha\varepsilon,\omega)-D_z F_\varepsilon(K_0,\theta,0,\omega)\big]d\alpha.  \notag
 \end{align} 
Under the assumption of hyperbolicity, the left-hand side of equation \eqref{M-random-N} is clearly an invertible operator. We thus can rewrite \eqref{M-random-N} as a fixed point equation for $\Delta_\varepsilon$:
 \begin{align}
\Delta_\varepsilon=(\mathscr{M}_0-Id)^{-1}\circ \mathscr{N}(\Delta_\varepsilon).
 \end{align} 
Denote by $\mathscr{G}=(\mathscr{M}_0-Id)^{-1}\circ \mathscr{N}:\mathscr{B}_{r}\to \mathcal{L}^\infty(\Omega,C(\TT^m,\RR^{2d}))$ the corresponding fixed point operator. We now proceed to show that $\mathscr{G}$ takes values in $\mathscr{B}_{r}$ for a suitable ${r}$, and that it is contractive.
\par
For simplicity, we assume that $D_zF_\varepsilon$ and $D_\varepsilon F_\varepsilon$ are Lipschitz with respect to the first and the third variables with a common Lipschitz constant $L$  (This makes sense since our model can be considered as a cut-off system). For $\Delta_\varepsilon\in \mathscr{B}_r$ we can find that
\begin{align}
    \|\mathscr{N}(\Delta_\varepsilon)\|
    \leqslant &\varepsilon|D_\varepsilon F_\varepsilon (K_0,\theta,0,\omega)|+(\varepsilon+\Delta_\varepsilon)\int_0^1 \alpha L(\Delta_\varepsilon+\varepsilon)d\alpha\notag\\
    \leqslant & \varepsilon C_0+\frac{1}{2}L(r+\varepsilon)^2.
\end{align}
Hence, with $A=\frac{1}{2}c_HL$, $B= c_HL\varepsilon $ and $C=c_H(\frac{1}{2}L\varepsilon^2+\varepsilon C_0)$, we have
 \begin{align}
 \|\mathscr{G}(\Delta_\varepsilon)\|\leqslant 
 &A r^2+Br+C.\notag
 \end{align}
In consequence, the operator $\mathscr{G}$ is well-defined in $\mathscr{B}_{r}$ as long as $(B-1)^2-4AC> 0$, that is $ 0<\varepsilon<{2L(c_H+C_0c_H^2)}^{-1},$ and $r\in[r^-,r^+]$ with 
 \begin{align}
r^\pm=\frac{1}{c_HL}-\varepsilon  \pm\frac{\sqrt{1-2\varepsilon L(c_H+C_0c_H^2)}}{c_HL}.\notag
 \end{align}
 It is not hard to check that the operator $\mathscr{G}$ is contracting in $\mathscr{B}_{r^\ast}$ for a certain $r^\ast\in (r^-\vee 0,\frac{1}{c_HL}-\varepsilon)$, and its Lipschitz constant is $2Ar^\ast+B=c_H L(r^\ast+\varepsilon)<1$. The results thus follows by the fixed point argument which preserves measurability.
\end{remark}

\subsection{Numerical algorithms for computing random invariant tori}

In this section, we discuss about the algorithm to solving the invariant equation for finding invariant tori. To this end, we should take full advantage of the standing property of the known object, that is, the hyperbolicity of the deterministic invariant torus. It leads us to think of the stable and unstable bundles, and construct adapted frames in which it is easier to measure the hyperbolicity. For simplicity,  we will consider that the invariant bundles are trivial (or easily trivializable), and so that we can construct global frames.

\subsubsection{Computing $K_k$}

Referring to \cite{Haro}, all we need is to assume that there exists a continuous matrix-valued map $P_0:\TT^d\to GL(\RR^n)$ such that
\begin{align}\label{reducibility}
P_0^{-1}(\theta+\alpha)M_0(\theta)P_0(\theta)=\Lambda_0(\theta),
 \end{align}
 where $\Lambda_0=\Lambda_0^S\times\Lambda_0^U:\;\TT^d\to L(\RR^{n_S})\times GL(\RR^{n_U})$ is a continuous matrix-valued map defining a block-diagonal linear skew-product $(\Lambda_0,\theta)$ in $ (\RR^{n_S} \times \RR^{n_U}) \times \TT^d$. In this way, the linear skew-product $(M_0,\theta)$ is hyperbolic, the matrix-valued map $P_0$ is regarded as an adapted frame, and \eqref{reducibility} is called a \textit{reducibility equation}. By supplementing the random invariance equation \eqref{random bundle invariant-2} with the reducibility equation \eqref{reducibility}, at each step the linear equation to be solved is, in some sense, diagonalized. We remark that such a reducibility technique has been wide used in many fields; see \cite{HdlL06a,HdlL06b,Haro,ZhangL18}. 

\begin{remark}
    We emphasize that both $P_0$ and $\Lambda_0$ are deterministic objects, and they are usually known by considering the deterministic system. In particularly, for the deterministic counterpart of the model \eqref{eq:model}, these two objects can be calculated by a Newton’s method \cite{ZhangL18}. 
\end{remark}

Keeping the equations \eqref{expansion-m} in mind, our goal is to compute $K_k$ $(1\leqslant k \leqslant l-1)$, assuming that $K_0$ is known and $K_i$ ($i=1,\cdots,k-1$)
have already been computed in previous steps. That is, we would like to produce a better approximation of the random invariant torus at each step of the algorithm, increasing the order by 1.

We first focus on the first order case, that is, 
\begin{align}\label{K-1}
M_0(\theta)K_1(\theta,\omega)-K_1(\theta+\alpha,\Phi_1\omega)=-\mathcal{R}_{0}(K_0)(\theta,\omega),
\end{align}
where $\mathcal{R}_{0}$ is given by \eqref{R-0}. By Theorem \ref{Thm2:expansion}, $K_1$ can be found by inverting the operator $\mathscr{M}_0- \text{Id}$. With the help of \eqref{reducibility}, we can obtain an explicit expression for $K_1$ via the following procedure. Under the adapted frame $P_0$, we define by
$$
\widetilde{K}_1(\theta,\omega)=P_0^{-1}(\theta)K_1(\theta,\omega)
$$
the first order correction, so that the new approximation is $K_0+\varepsilon K_1=K_0+\varepsilon P_0\widetilde{K}_1$. By multiplying $P_0^{-1}(\theta+\alpha)$ on the both sides of \eqref{K-1}, we have
 \begin{align}\label{K-1-2}
P_0^{-1}(\theta+\alpha)M_0(\theta)P_0(\theta)\widetilde{K}_1(\theta,\omega)-\widetilde{K}_1(\theta+\alpha,\Phi_1\omega)=-\widetilde{\mathcal{R}}_{0}(\theta,\omega),
\end{align}
with $\widetilde{\mathcal{R}}_{0}=P_0^{-1}(\theta+\alpha)\mathcal{R}_{0}$ being a known function (i.e., independent of $K_1$). Taking the reducibility equation \eqref{reducibility} in account, we infer that 
 \begin{align}\label{K-1-3}
\Lambda_0(\theta)\widetilde{K}_1(\theta,\omega)-\widetilde{K}_1(\theta+\alpha,\Phi_1\omega)=-\widetilde{\mathcal{R}}_{0}(\theta,\omega).
\end{align}
Therefore, we only need to find $\widetilde{K}_1$ given $\widetilde{\mathcal{R}}_{0}$. Moreover, by splitting to the stable and unstable direction, equation \eqref{K-1-3} can be rewritten as
\begin{equation}\label{WSU-1}
\begin{split}
\left \{
  \begin{array}{ll}
\widetilde{K}_1^S(\theta,\omega)&=\Lambda_0^S(\theta-\alpha)\widetilde{K}_1^S(\theta-\alpha,\Phi_{-1}\omega)+\widetilde{\mathcal{R}}_{0}^S(\theta-\alpha,\Phi_{-1}\omega), \\
\widetilde{K}_1^U(\theta,\omega)&=\Lambda_0^U(\theta)^{-1}\big[\widetilde{K}_1^U(\theta+\alpha,\Phi_1\omega)-\widetilde{\mathcal{R}}_{0}^U(\theta,\omega)\big].
\end{array}
   \right.
  \end{split}
\end{equation}
It is clear that the above two equations are contractions which enable us to write $\widetilde{K}_1$ in terms of converging infinite
series.

Denoting $T_{\alpha,\Phi_1}^{j}=(\theta+j\alpha,\Phi_j\omega)$ for $j\in \mathbb{Z}$, the equations \eqref{WSU-1} can be solved explicitly:
\begin{equation}
\begin{split}
\left \{
  \begin{array}{ll}
\widetilde{K}_1^S&=\sum_{N=0}^{\infty}\left[\prod_{j=1}^{N}\Lambda_0^{S}\circ T_{\alpha,\Phi_1}^{-j}\right]\widetilde{\mathcal{R}}_{0}^{S}\circ T_{\alpha,\Phi_1}^{-(N+1)}, \\
\widetilde{K}_1^U&=-\sum_{N=0}^{\infty}\left(\Lambda_0^{U}\right)^{-1}\left[\prod_{j=0}^{N}\Lambda_0^{U}\circ T_{\alpha,\Phi_1}^{j}\right]\widetilde{\mathcal{R}}_{0}^{U}\circ T_{\alpha,\Phi_1}^{N}.  
\end{array}
   \right.
  \end{split}
\end{equation}
We thus have $K_1^S=P_0\widetilde{K}_1^S$ and $K_1^U=P_0\widetilde{K}_1^U$. Noticing that both $\Lambda_0$ and $P_0$ are deterministic,  $K_1$ has the same distribution with $\mathcal{R}_0$ for fixed $\theta$. Furthermore, by RDE \eqref{F-RDE}, $K_1$ is indeed Gaussian. 

More interesting phenomenon occurs when we consider higher order terms. Keeping \eqref{expansion-m} in mind, higher order terms of $K$ should be readily
obtainable as well, since all the right hand side terms of the $k$-th order
equation only depends on lower order terms. We can get exactly the same formulae for $K_k$ ($k\geqslant 2$) as $K_1$ except for a different $\mathcal{R}_k$ term. It should be noted that the formulae for $\mathcal{R}_k$ involve stochastic processes given in RDEs \eqref{F-RDE-k}. The distributions of $K_k$ ($k\geqslant 2$) corresponds to a polynomial of Gaussian distribution. 

\subsubsection{Computing Lyapunov exponents}

It is natural to expand the definition of reducibility for random manifold. We say that the invariant torus $\mathcal{K}$ is reducible if the linear random skew-product $(M,\alpha,\Phi)$ 
is \textit{reducible} to a random matrix $$\Lambda(\theta,\omega)=\text{blockdiag}\big[\Lambda^S(\theta,\omega),\Lambda^U(\theta,\omega)\big].$$ That is, there is change of variables $P:\TT^m\times \Omega\to GL(\RR^{2d})$, such that the \textit{random reducibility equation}
 \begin{align}\label{random reducibility}
P(\theta+\alpha,\Phi_1\omega)^{-1}M(\theta,\omega) P(\theta,\omega)=\Lambda(\theta,\omega) 
\end{align}
is satisfied, where $M(\theta,\omega)=D_z F_\varepsilon(K_\varepsilon,\theta,\varepsilon,\omega)$. We thus define by $(\Lambda,\alpha,\Phi)$ a block-diagonal linear skew-product in $ (\RR^{{2d}_S} \times \RR^{{2d}_U}) \times \TT^m\times\Omega$, with the matrix-valued map $\Lambda=\Lambda^S\times\Lambda^U:\;\TT^m\times\Omega\to L(\RR^{{2d}_S})\times GL(\RR^{{2d}_U})$.

Again, following the approach
of perturbation theory, we consider the expansions: 
\begin{align}
  P&=P_0+\varepsilon P_1+\cdots+\varepsilon^k P_k+\cdots,\notag\\ 
  \Lambda&=\Lambda_0+\varepsilon\Lambda_1+\cdots+\varepsilon^k\Lambda_k+\cdots.\notag 
\end{align}
Clearly, the first order terms $P_0$ and $\Lambda_0$ are known, and our aim is to find higher order terms. To obtain the results of first order, we consider the new approximations
$$
P_0+\varepsilon P_1=P_0+\varepsilon P_0Q_1\;\;\text{and}\;\;\Lambda_0+\varepsilon\Lambda_1.
$$
Note that, by eliminating quadratically small terms, we can calculate that 
\begin{align}
0=&(P_0+\varepsilon P_1)^{-1}(\theta+\alpha,\Phi_1\omega)D_z F_\varepsilon\big(K_0+\varepsilon K_1,\theta,\varepsilon,\omega\big) (P_0+\varepsilon P_1)(\theta,\omega)-(\Lambda_0+\varepsilon\Lambda_1)(\theta,\omega) \notag\\
=&\big[(\text{Id}- \varepsilon Q_1)P_0^{-1}\big](\theta+\alpha,\Phi_1\omega)D_z F_\varepsilon\big(K_0+\varepsilon K_1,\theta,\varepsilon,\omega\big) \big[P_0(\text{Id}+ \varepsilon Q_1)\big](\theta,\omega)\notag\\
&-\big[\Lambda_0(\theta)+\varepsilon \Lambda_1(\theta,\omega)\big] \notag\\
=&P_{0}^{-1}(\theta+\alpha)D_z F_\varepsilon\big(K_0,\theta,0,\omega\big)P_{0}(\theta)-\Lambda_{0}(\theta)-\varepsilon \Lambda_1(\theta,\omega)+\varepsilon \mathcal{E}(K_0,K_1)(\theta,\omega)\notag\\
&- \varepsilon Q_1(\theta+\alpha,\Phi_1\omega)P_{0}^{-1}(\theta+\alpha)D_z F_\varepsilon\big(K_0,\theta,0,\omega\big)P_{0}(\theta) \notag\\
&+\varepsilon P_{0}^{-1}(\theta+\alpha)D_z F_\varepsilon\big(K_0,\theta,0,\omega\big)P_{0}(\theta) Q_1(\theta,\omega)\notag
\end{align}
with $\mathcal{E}=D_z^2F_\varepsilon(K_0,\theta,0,\omega)K_1(\theta,\omega)+D_\varepsilon D_zF_\varepsilon(K_0,\theta,0,\omega)$.
Hence, by \eqref{reducibility}, we turn to dealing with the following equation
 \begin{align}\label{reducibility-redefined}
\Lambda_0(\theta) Q_1(\theta,\omega)- Q_1(\theta+\alpha,\Phi_1\omega)\Lambda_0(\theta)- \Lambda_1 (\theta,\omega)=-\mathcal{E}(K_0,K_1)(\theta,\omega).
 \end{align}
We adopt the block matrix notation
$$
Q=\begin{pmatrix} Q^{SS}& Q^{SU} \\ Q^{US} & Q^{UU}\end{pmatrix}.
$$
Then, by writing \eqref{reducibility-redefined} in block form, we have
\begin{align}
\Lambda_0^SQ_1^{SS}-(Q_1^{SS}\circ T_{\alpha,\phi_1})\Lambda_0^{S}&=-\mathcal{E}^{SS}+\Lambda_1^S, \label{A1}\\
\Lambda_0^SQ_1^{SU}-(Q_1^{SU}\circ T_{\alpha,\phi_1})\Lambda_0^{U}&=-\mathcal{E}^{SU}, \label{A2}\\
\Lambda_0^UQ_1^{US}-(Q_k^{US}\circ T_{\alpha,\phi_1})\Lambda_0^{S}&=-\mathcal{E}^{US}, \label{A3}\\
\Lambda_0^UQ_1^{UU}-(Q_k^{UU}\circ T_{\alpha,\phi_1})\Lambda_0^{U}&=-\mathcal{E}^{UU}+\Lambda_1^U. \label{A4}
\end{align}
It is clear that equations \eqref{A2} and \eqref{A3} can be solved by using the contraction principle:
\begin{align}
Q_1^{SU}&=\big[ (\Lambda_0^{S}\circ T_{\alpha,\phi_1}^{-1})(Q_k^{SU}\circ T_{\alpha,\phi_1}^{-1})+(\mathcal{E}^{SU}\circ T_{\alpha,\phi_1}^{-1})\big]\big( \Lambda_0^{U}\circ T_{\alpha,\phi_1}^{-1}\big)^{-1}, \\
Q_1^{US}&=\big(\Lambda_0^U\big)^{-1}\big[(Q_k^{US}\circ T_{\alpha,\phi_1})\Lambda_0^{S}-\mathcal{E}^{US}\big].
\end{align}
As for equations  \eqref{A1} and \eqref{A4}, we note that, for each $\omega\in\Omega$, the solutions to them are not unique. For convenience, we choose $Q_1^{SS}=0$ and $Q_1^{UU}=0$. 
The geometrical meaning of such a choice of $Q_k$ is that, for each realization of the noise, we can modify the stable and unstable frames in the complementary directions; See \cite{Haro,ZhangL18}. In this way, we conclude that, by taking $P_1=P_0 Q_1$, the first-order expansion for $\Lambda$ satisfies
\begin{align}
\Lambda_1^S=\mathcal{E}^{SS} \;\;\text{and}\;\; \Lambda_1^U=\mathcal{E}^{UU}.
\end{align}

\section*{Acknowledgments}
The authors thank Prof. Rafael de la Llave, Prof. Jinqiao Duan, Dr. Hongyu Cheng and Dr. Ying Chao for helpful discussions. PW would especially like to express his gratitude to Prof. Rafael de la Llave for hospitality during his visit to Georgia Institute of Technology.

\bibliographystyle{alpha}
\bibliography{ref-tst}

\section*{Appendix}

\addcontentsline{toc}{section}{Appendix}

\subsection*{A. The cut-off technique.}\label{cut-off}



For technical reasons as mentioned in Section 2.1, we may need to introduce the following approximate SDE (i.e., a cut-off system):
 \begin{equation}   
  \label{eq:model-cut}
  \begin{split}
  \left \{
  \begin{array}{ll}
   d{x}_t^n&=v_t^ndt,\\
   d{v}_t^n&=-\gamma \chi_n(v_t^n)dt-  \nabla U_n (x_t^n-E(t))dt+\varepsilon dW_t,
  \end{array}
   \right.
    \end{split}
  \end{equation}
with initial values $x_0^n=x_0$ and $v_0^n=v_0$, where for $n>0$, $\chi_n$ is a smooth function satisfying
$$
\chi_n(x)=\left \{
  \begin{array}{ll}
   1,&\;\;|x|\leqslant n; \\
   0,&\;\;|x|>2n,
  \end{array}
   \right.
$$
and $U_n(x)=U(x)\chi_n(x)$. Let $D_n=\{ |z|<n\}\subset \RR^{2d}$ be a bounded domain with $n>R>0$. We define
 \begin{equation}   
  \label{stopping time}
\tau_{n,R}=\{t\geqslant0: \sup_{|z_0|\leqslant R} |z_t(z_0)| \geqslant n \}
  \end{equation}
as the first exit time for the process $z_t$ of \eqref{eq:model} from $D_n$. It is clear that
 \begin{equation}   
  \label{local solution}
z_t(z_0)=z_t^n(z_0):=\big(x_t^n(x_0),v_t^n(v_0)\big)^T,\;\;\forall t\in [0,\tau_{n,R}).
\end{equation}
Moreover, based on the lemma below, the cut-off system and the original system share the same asymptotic dynamics, except for rare events (i.e., events of small probability). Consequently, a random invariant manifold of the cut-off system serves as a locally random invariant manifold for the original system.

\begin{remark}
The cut-off function $\chi$ in \eqref{eq:model-cut} is usually not unique. Consequently, the random invariant manifold we construct will not be unique either, even if it becomes unique once the cut-off equations are specified. Moreover, while the existence of a cut-off function is a sufficient condition for working with the cut-off equations, it is by no means necessary. In many cases, one can leverage the structure of the nonlinearity to construct the cut-off equations; see \cite{Cheng2019,KanDuan2013}.
\end{remark}

\begin{lemma} \label{lem1}
Consider the SDE \eqref{Cauchy-SDE} with $z(0)=z_0$. Then, for $R>0$, 
 \begin{equation}   
 \label{lem-PP}
\lim_{n\to\infty}\PP\left[\sup_{|z_0|\leqslant R} |z_t(z_0)| \geqslant n \right]=0.
  \end{equation}
\end{lemma}

\begin{proof}
The main idea of this proof is to construct a Lyapunov function. This technique has been used in the case $E(t)\equiv 0$ (see Section 3 in \cite{Mattingly2002}) and for other models (see Section 3.4 in \cite{Khasminskii2012}).

Note that $E(t)$ is bounded and $U(x)$ is  bounded below. In particular, given \eqref{eq:1} and \eqref{eq:2}, we have
$$
|E(t)|\leqslant m \text{ and } U(x)\geqslant \frac{1}{4\delta} \;\;(\delta<0).
$$
Motivated by the case of $E(t)\equiv 0$ \cite{Mattingly2002}, we define a Lyapunov function
\begin{equation}
\label{Lyapunov function}
\mathcal{H}(z,t)=\frac{1}{2}|v|^2+U(x-E(t))+\frac{1}{2}\gamma\langle x,v\rangle +\frac{\gamma^2}{4}|x|^2+C,
\end{equation}
where $C$ is a positive constant satisfying $C>-\inf_x U(x)$. It is clear that
\begin{equation}
\label{Lyapunov-1}
\mathcal{H}(z,t)>\frac{1}{8}|v|^2+\frac{\gamma^2}{12}|x|^2\geqslant 0, \;\;\forall t>0, \notag
\end{equation}
and
\begin{equation}
\label{Lyapunov-2}
\inf_{t>0,|z|>n}\mathcal{H}(z,t) \to \infty,\;\;\text{as}\;\;n\to\infty.
\end{equation}
One the other hand, we claim that
\begin{equation}
\label{Lyapunov-3}
\mathscr{A}\mathcal{H}(z,t)\leqslant - C_0\mathcal{H}(z,t) +C_1,
\end{equation}
where $C_0,C_1>0$ are positive constants and $\mathscr{A}$ is the generator of the SDE \eqref{Cauchy-SDE} given by
\begin{equation}
\label{Generator-SDE}
\mathscr{A}= Az\cdot (\nabla_z, \partial_t) - \nabla U (x-E(t))\cdot (\nabla_v,\partial_t)+\frac{1}{2} \varepsilon^2\Delta_v.
\end{equation}
In fact, this follows from a modification of Lemma 3 in \cite{Mattingly2002} by taking the bounded function $E(t)$ into account. As the proof for this claim is standard, we omit the details here.

By \eqref{Lyapunov-3}, Dynkin's formula and Gronwall’s inequality, we thus have
\begin{align}
\label{Dynkin}
\mathbb{E}\mathcal{H}(z,t\wedge \tau_{n,R})
=&\mathbb{E}\mathcal{H}(z_0,0)+\mathbb{E}\int_0^{t\wedge \tau_{n,R}}\mathscr{A}\mathcal{H}(z,s) ds \notag\\
\leqslant&e^{-C_0(t\wedge \tau_{n,R})}\mathbb{E}\mathcal{H}(z_0,0)+\frac{C_1}{C_0}\left[1-e^{-C_0 ( t\wedge \tau_{n,R})}\right].
\end{align}
Applying the Chebyshev’s inequality, we derive the estimate
\begin{align}
\label{Chebyshev}
\PP\left[\sup_{|z_0|\leqslant R} |z_t(z_0)| \geqslant n \right]\leqslant\frac{e^{-C_0 t}\mathbb{E}\mathcal{H}(z_0,0)+\frac{C_1}{C_0}\left[1-e^{-C_0 t}\right]}{\inf_{t>0,|z|>n}\mathcal{H}(z,t)}.
\end{align}
The result follows by letting $n\to\infty$ and making use of \eqref{Lyapunov-2}.
\end{proof}
\begin{remark}
An approximation SDE in the form of \eqref{eq:model-cut} can be also proposed to prove the the existence and uniqueness of strong solutions. In fact, once we have the local solution $z_t=\{z_t\}_{t\in[0,\tau_{n,R})}$ in \eqref{local solution}, we only need to show that $\lim_{n\to\infty} \tau_{n,R}=\infty$, a.e. Lemma \ref{lem1} ensures that $ \tau_{n,R}\to\infty$ in probability, and it means that there exists a subsequent $\{\tau_{n_k,R}\}$ such that $\tau_{n_k,R}\to\infty$ a.e. Based on the fact that $\tau_{n,R}$ is increasing with respect to $n$, we can thus conclude that the local solution \eqref{local solution} is a global one.
\end{remark}


\end{document}